\documentclass[12pt, a4 paper]{amsart}
\usepackage{amssymb,amsmath,amsthm,amsfonts,amscd}
\usepackage{xcolor}
\usepackage{tikz}
\usepackage{fix-cm}
\usepackage{datetime}
\usepackage{comment}

\usepackage[colorlinks,citecolor=blue]{hyperref}

\newtheorem{theorem}{Theorem}[section]
\newtheorem{proposition}[theorem]{Proposition}
\newtheorem{lemma}[theorem]{Lemma}
\newtheorem{corollary}[theorem]{Corollary}
\theoremstyle{definition}
\newtheorem{definition}[theorem]{Definition}
\newtheorem{example}[theorem]{Example}

\newtheorem{remark}[theorem]{Remark}

\def \C{\mathbb{C}}
\def \T{\mathbb{T}}

\def \N{\mathbb{N}}
\newcommand{\clb}{\mathcal{B}}

\newcommand{\cle}{\mathcal{E}}

\newcommand{\clk}{\mathcal{K}}

\newcommand{\clf}{\mathcal{F}}

\newcommand{\na}{\mathcal{N}}
\newcommand{\ma}{\mathcal{M}}
\newcommand{\an}{\mathcal{AN}}
\newcommand{\am}{\mathcal{AM}}
\newcommand{\es}{\sigma_{\text{ess}}}
\newcommand{\vp}{\varphi}
\newcommand{\red}{\mathcal{R}}
\newcommand\restr[2]{\ensuremath{\left.#1\right|_{#2}}}

\numberwithin{equation}{section}



\author[N. Puspendu]{Nag Puspendu}
\address{Department of Mathematics, IIT Hyderabad, 
Kandi, Sangareddy, Telangana, India, 502\,284.}
\email{ma23resch11003@iith.ac.in; puspomath@gmail.com}

\author[G. Ramesh]{Golla Ramesh}
\address{Department of Mathematics, IIT Hyderabad, 
Kandi, Sangareddy, Telangana, India, 502\,284.}
\email{rameshg@math.iith.ac.in}

\begin{document}
\title[Minimum Attaining Operators on Reducing Subspaces]{Minimum Attaining Operators on Reducing Subspaces: Spectral Structure and Density}

\maketitle
\begin{abstract}
In this article, we introduce and study the class $\mathcal{M}_r(H)$ of
minimum attaining operators on a separable Hilbert space $H$
whose restrictions to every nonzero reducing subspace are minimum
attaining. This class lies between the classes of absolutely minimum
attaining operators and minimum attaining operators. We establish a
spectral characterization of positive operators in
$\mathcal{M}_r(H)$: a positive operator $T$ belongs to
$\mathcal{M}_r(H)$ if and only if $T=\sum_{\lambda\in\Lambda}\lambda P_\lambda$, where $\Lambda\subseteq[0,\|T\|]$ is a countable well-ordered set,
$(P_\lambda)_{\lambda\in\Lambda}$ is a family of mutually orthogonal
nonzero projections, and $\sum_{\lambda\in\Lambda}P_\lambda=I$ in the
strong operator topology.

We also establish a connection between
$\mathcal{M}_r(H)$ and the class of operators attaining
their norm on every nonzero reducing subspace. Further, we obtain
stability results under suitable finite-rank and compact perturbations, and prove that $\mathcal{M}_r(H)$
is norm dense in $\mathcal{B}(H)$. We also establish the
norm density of operators in this class having a nontrivial invariant
half-space. Finally, we characterize normal operators
in $\mathcal{M}_r(H)$, obtain a corresponding
characterization for multiplication operators on $L^2(\mathbb{T})$,
and characterize the condition
$\mathcal{R}_T=\mathcal{R}_{|T|}$ in terms of the von Neumann algebra
generated by $|T|$, showing that this condition forces $T$ to be
normal.
\end{abstract}

\section{introduction}
The study of extremal properties of bounded linear operators is a fundamental theme in functional analysis and operator theory. A classical example is norm attainment, which concerns whether an operator attains its norm on the unit sphere. The corresponding lower-extremal notion is described by the minimum modulus. For $T\in\mathcal B(H_1,H_2)$, the minimum modulus of $T$ is defined by
$$
m(T):=\inf\{\|Tx\|:x\in S_{H_1}\},
$$
where $S_{H_1}$ denotes the unit sphere of the Hilbert space $H_1$. The minimum modulus is closely related to bounded below operators, closed range, and the spectral properties of the positive operator $T^*T$; when $H_1=H_2$, it is also related to invertibility. A systematic study of the minimum modulus and its role in spectral theory was carried out by Gindler and Taylor \cite{GindlerTaylor}.

An operator $T\in\mathcal B(H_1,H_2)$ is called minimum attaining if there exists $x_0\in S_{H_1}$ such that
$$
\|Tx_0\|=m(T).
$$

The systematic study of minimum attaining operators was initiated by Carvajal and Neves \cite{Carvajal 2}, who also introduced the corresponding absolute version of minimum attainment. More precisely, an operator $T\in\mathcal B(H_1,H_2)$ is called absolutely minimum attaining, or an $\mathcal{AM}$-operator, if $T|_M:M\longrightarrow H_2$ is minimum attaining for every nonzero closed subspace $M$ of $H_1$. If $\mathcal M(H)$ denotes the class of minimum-attaining operators on $H$, then $\mathcal{AM}(H)\subseteq\mathcal M(H)$. The $\mathcal{AM}$ property is substantially stronger than ordinary minimum attainment, since it requires minimum attainment to persist after restriction to every nonzero closed subspace. This condition has led to strong spectral and structural results, particularly for positive operators; see \cite{AM Cha,NBGRAM}. The absolute minimum attaining property has also been investigated for operator matrices; see \cite{Block AN 2026}.

The requirement in the definition of an $\mathcal{AM}$-operator is imposed on all nonzero closed subspaces. From the operator-theoretic point of view, however, not every closed subspace is naturally associated with the operator. A more intrinsic family of subspaces is provided by the reducing subspaces. Recall that a closed subspace $M$ of $H$ reduces $T\in\mathcal B(H)$ if both $M$ and $M^\perp$ are invariant for $T$. Equivalently, if $P_M$ denotes the orthogonal projection onto $M$, then $P_M$ commutes with $T$. In this case,
$$
H=M\oplus M^\perp,\qquad
T=T|_M\oplus T|_{M^\perp}.
$$
Thus, reducing subspaces correspond to operator decompositions and are closely tied to the spectral structure of the operator.

The idea of imposing an attainment condition only on reducing subspaces has already appeared in the theory of norm attaining operators. In \cite{betaclass}, a subclass of norm attaining operators, referred to as the $\beta$-class, was introduced by requiring the restriction of an operator to every nonzero reducing subspace to attain its norm. Subsequent investigations of this class can be found in \cite{Beta AFA,Beta Jot}. These works demonstrate that an extremal condition imposed on reducing subspaces provides a natural framework between ordinary and absolute attainment.

This observation motivates the following question: can one obtain a meaningful intermediate class of minimum attaining operators by requiring minimum attainment only on the reducing subspaces of the operator? We give an affirmative answer. For $T\in\mathcal B(H)$, let $\mathcal R_T$ denote the collection of all reducing subspaces for $T$. We define
$$
\mathcal M_r(H):=
\left\{
T\in\mathcal B(H):
T|_M\in\mathcal M(M)
\text{ for every nonzero }M\in\mathcal R_T
\right\}.
$$
We refer to the operators in $\mathcal M_r(H)$ as minimum attaining on reducing subspaces. By definition,
$$
\mathcal{AM}(H)\subseteq\mathcal M_r(H)\subseteq\mathcal M(H).
$$
Both inclusions are proper, and hence $\mathcal M_r(H)$ forms a genuinely new intermediate class between the absolutely minimum-attaining and minimum-attaining operators.

The significance of $\mathcal M_r(H)$ lies in the way it connects minimum-modulus attainment with the reducing-subspace structure of an operator. The condition defining $\mathcal M_r(H)$ is weaker than absolute minimum attainment because no requirement is imposed on arbitrary closed subspaces. On the other hand, it is stronger than ordinary minimum attainment, since every nonzero reducing component must itself attain its minimum modulus. In particular, every irreducible minimum-attaining operator belongs to $\mathcal M_r(H)$. The class also contains several familiar classes of operators, including idempotents, orthogonal projections, partial isometries, and certain Toeplitz operators. Thus, the definition captures a structural form of minimum attainment that is compatible with the operator's reducing decompositions.

The purpose of this paper is to develop the spectral and structural theory of $\mathcal M_r(H)$. Our results are developed in parallel with several known results for $\beta$-class operators. The arguments presented here are based on the defining property of $\mathcal M_r(H)$, namely, minimum attainment on every nonzero reducing subspace, and on the characterizations established in this paper. We give independent proofs of our results.

We first establish the fundamental properties of $\mathcal M_r(H)$ and study its behaviour under unitary equivalence, adjoints, and the modulus. We then obtain spectral characterizations, beginning with self-adjoint and positive operators. A central feature of the positive case is the appearance of bounded countable well-ordered subsets of $[0,\infty)$. More precisely, a positive operator $T$ belongs to $\mathcal M_r(H)$ if and only if it admits a spectral representation
$$
T=\sum_{\lambda\in\Lambda}\lambda P_\lambda,
$$
where $\Lambda\subseteq[0,\|T\|]$ is a countable well-ordered set and $(P_\lambda)_{\lambda\in\Lambda}$ is a family of mutually orthogonal nonzero projections whose sum is $I$ in the strong operator topology.

The well-ordering condition is a direct consequence of the requirement that minimum attainment persist on reducing subspaces. Indeed, if a positive operator has eigenvalues
$$
\lambda_1>\lambda_2>\lambda_3>\cdots\longrightarrow\lambda,
$$
where $\lambda$ is not among the eigenvalues, then the closed linear span of the corresponding eigenspaces is a reducing subspace on which the minimum modulus is $\lambda$, but this minimum is not attained. Thus, such an operator cannot belong to $\mathcal M_r(H)$. This phenomenon underlies the spectral characterization obtained in the positive case.

It is known that the set $\mathcal M(H)$ of all minimum-attaining operators is dense in $\mathcal B(H)$ with respect to the operator norm; see \cite{SHK:RG2018} for further details. The spectral characterization established in this paper enables us to obtain a corresponding density result for the more restrictive class $\mathcal M_r(H)$. We show that self-adjoint operators with finite spectrum belong to $\mathcal M_r(H)$ and use this fact, together with the polar decomposition, to prove that
$$
\overline{\mathcal M_r(H)}^{\,\|\cdot\|_{\mathrm{op}}}
=\mathcal B(H).
$$
Furthermore, we prove that the subclass of $\mathcal M_r(H)$ consisting of operators with a nontrivial invariant half-space is also norm dense in $\mathcal B(H)$. These results show that the requirement of minimum attainment on every nonzero reducing subspace is compatible with norm approximation, even when the approximating operators are required to possess a nontrivial invariant half-space.

We further study the behaviour of $\mathcal M_r(H)$ under natural operator-theoretic constructions, including finite-rank and compact perturbations and direct sums. The interaction with the modulus is particularly useful for quasinormal and normal operators. For a quasinormal operator $T$, membership in $\mathcal M_r(H)$ is determined by the corresponding property of the modulus. In the normal case, we obtain a characterization together with a structural representation by orthogonal direct sums of scalar multiples of unitary operators indexed by a bounded countable well-ordered set. As a concrete application, we characterize multiplication operators on $L^2(\mathbb T)$ whose symbols give rise to operators in $\mathcal M_r(L^2(\mathbb T))$.

Finally, we examine the condition $\mathcal R_T=\mathcal R_{|T|}$
for operators $T\in\mathcal B(H)$. Using the polar decomposition $T=V|T|$, we characterize this condition in terms of the partial isometry $V$ and the von Neumann algebra generated by $|T|$. In particular, we show that
$$
\mathcal R_T=\mathcal R_{|T|}
\quad\Longleftrightarrow\quad
V\in W^*(|T|),
$$
equivalently, $W^*(T)=W^*(|T|)$. We further establish that this condition forces $T$ to be normal, although normality alone does not imply $\mathcal R_T=\mathcal R_{|T|}$.

The main theme of this paper is the interplay between minimum attainment, reducing subspaces, and spectral structure. The results provide spectral characterizations for important classes of operators in $\mathcal M_r(H)$, establish stability and density properties, and clarify the relationship between the reducing subspaces of an operator and those of its modulus.

\section{Preliminaries}
Throughout this article, we deal with infinite dimensional separable Hilbert spaces over $\C$, and denote them by $H,H_1,H_2$ etc. For a Hilbert space $H$, we denote its unit sphere by $S_H:=\{x\in H:\|x\|=1\}$. The class of all bounded linear operators from $H_1$ into $H_2$ is denoted by $\clb(H_1,H_2)$. For $T\in \clb(H_1,H_2)$, the \emph{minimum modulus} of $T$ is defined by $m(T):=\inf\{\|Tx\| : x\in S_{H_1}\}$.

We say $T\in\clb(H_1,H_2)$ is \emph{norm attaining}, if $\|Tx_0\|=\|T\|$ for some $x_0\in S_{H_1}$. We denote this class by $\na(H_1,H_2)$, in particular $\na(H):=\na(H,H)$. The operator $T$ is said to be \emph{minimum attaining} if there exists $y_0\in S_{H_1}$ such that $\|Ty_0\|=m(T)$. The class of all minimum attaining operators in $\clb(H_1,H_2)$ is denoted by $\ma(H_1,H_2)$, in particular $\ma(H):=\ma(H,H)$.

For every nonzero closed subspace $M$ of $H_1$, (i) if $T|_M\in \na(M,H_2)$, then $T$ is called \textit{absolutely norm attaining} ($\an$) operator; (ii) if $T|_M\in\ma(M,H_2)$, then $T$ is called \textit{absolutely minimum attaining} ($\am$) operator. We refer to \cite{Carvajal 1,Carvajal 2,AM Cha,NBGRAM} for details about these classes.

We denote the class of finite rank operators and the class of compact operators by $\clf(H_1,H_2)$ and $\clk(H_1,H_2)$, respectively. For $H=H_1=H_2$, we use the notation $\clf(H):=\clf(H,H)$ and $\clk(H):=\clk(H,H)$. We write $\clb(H)_{sa}$ and $\clb(H)_{+}$ for the sets of all self-adjoint and positive operators on $H$, respectively.

For $T\in \clb(H)$, the \textit{spectrum} of $T$ is defined by $$\sigma(T):=\{\lambda\in \C : T-\lambda I \text{ is not invertible}\},$$ and the \textit{point spectrum} of $T$ is defined by $$\sigma_p(T):=\{\lambda\in\C : T-\lambda I \text{ is not one-one}\}.$$
An operator $T\in\clb(H)$ is called \textit{Fredholm} if $R(T)$ is closed and $N(T)$, $N(T^*)$ are finite dimensional. The \textit{essential spectrum} of $T$ is defined by $$\es(T):=\{\mu \in \C : T-\mu I \text{ is not Fredholm}\}.$$
For basic results concerning the spectrum and spectral theory of
bounded linear operators, we refer to \cite{Gohberg}.

A closed subspace $M$ of $H$ is said to be \textit{invariant} for $T\in \clb(H)$ if $T(M)\subseteq M$. Further, $M$ is called \textit{reducing} for $T$ if it is invariant for $T$ as well as $T^*$. We denote the set of all invariant and reducing subspaces for $T$ by $Lat_T$ and $\red_T$, respectively.
Let $P_j$ denote the orthogonal projection of $H$ onto $H_j$ for $j=1,2$. For any $T\in \clb (H)$, its block matrix representation with respect to the decomposition $H= H_1\oplus H_2$ is given by 
\begin{equation}
T=\begin{bmatrix}
    A_{11} & A_{12} \\
    A_{21} & A_{22}
\end{bmatrix},
\end{equation}
where $A_{ij}=P_iTP_j|_{H_j}\in \clb(H_j, H_i)$ for $i,j \in\{1,2\}$. If $H_1,H_2 \in Lat_T$, then $A_{12}=0$ and $A_{21}=0$, and we write $T=A_1\oplus A_2$. For basic properties of invariant and reducing subspaces, we refer to \cite{Conway,Inv1973}.

\smallskip
Let $H=\bigoplus_{n=1}^\infty H_n$ be direct sum of Hilbert spaces $H_n$, defined by $$H=\{(x_n) : x_n \in H_n, n\in\N \text{ with } (\|x_n\|)\in \ell^2(\N)\}.$$ 
Then $H$ is a Hilbert space with respect to the inner product $$\langle (x_n), (y_n) \rangle_{H}=\sum_{n\in \N} \langle x_n,y_n \rangle_{H_n},\quad (x_n),(y_n)\in H.$$
For $T_n\in \clb(H_n)$, the operator $T=\bigoplus_{n\in \N} T_n : H\longrightarrow H$ is defined by 
\begin{equation}\label{direct sum operators}
    T(x_1,x_2,\cdots)=(T_1x_1,T_2x_2,\cdots),\quad x_n\in H_n, n\in \N.
\end{equation}
Note that $T\in \clb(H)$ if and only if $\sup_{n\in\N} \{\|T_n\|\}<\infty$, and the norm of $T$ is given by 
\begin{equation}
\|T\|=\sup_{n\in\N} \|T_n\|
\end{equation}
(see \cite[Exercise 12, Page 30]{Conway}).

\section{Minimum Attainment on Reducing Subspaces}
In this section, we introduce a class of operators obtained by requiring
minimum attainment on the reducing subspaces naturally associated with
the operator. For $T\in\clb(H)$, let $\red_T$ denote the family of all
reducing subspaces of $T$. We define the class of operators that are
minimum attaining on every nonzero reducing subspace as follows.
\begin{definition}
\begin{equation*}
\ma_r(H):=
\left\{
T\in\clb(H):
\restr{T}{M}\in\ma(M)
\text{ for every nonzero }M\in\red_T
\right\}.
\end{equation*}
\end{definition}
It is immediate from the definition that
$$
\am(H)\subseteq \ma_r(H)\subseteq \ma(H),
$$
and we will subsequently present concrete examples showing that both inclusions are strict.

We next investigate the basic properties of $\ma_r(H)$. In particular, we study its relationship with the adjoint and the modulus, its behavior under unitary equivalence. We also identify some important classes of operators contained in $\ma_r(H)$, such as idempotents, partial isometries, and Toeplitz operators etc.

For $T\in\clb(H)$, we often use the notation $T_M :=\restr{T}{M}: M\longrightarrow M$, where $M$ is a reducing subspace for $T$. We denote the set of all positive operators in $\ma_r(H)$ by $\ma_r(H)_+$.

 \begin{definition}
     Let $T\in\clb(H)$. Then $T$ is called \emph{reducible} if $T$ has a nonzero proper reducing subspace. Otherwise $T$ is called an \emph{irreducible} operator.
 \end{definition}
We refer to \cite{Halmos} for details on irreducible operators.

 \begin{remark}
    Let $T\in \clb(H)$ be an irreducible operator, that is, $\red_T=\{\{0\},H\}$. If $T\in\ma(H)$, then $T\in\ma_r(H)$.
\end{remark}

 \begin{proposition}\cite[Proposition 3.1]{AM Cha}\label{eigenvalue}
   Let $T\in \clb(H)_{sa}$. Then $T\in \ma(H)$ if and only if $m(T)$ or $-m(T)$ is an eigenvalue of $T$.  
 \end{proposition}

Next, we illustrate with the following example that $\ma_r(H)$ is strictly contained in $\ma(H)$.
 
\begin{example}
    Let $T: \ell^2(\N)\longrightarrow \ell^2(\N)$ be defined by 
$$
T(x_1,x_2,x_3,\cdots)=\left(0,(\frac{1}{2}+\frac{1}{2})x_2, (\frac{1}{2}+\frac{1}{3})x_3,(\frac{1}{2}+\frac{1}{4})x_4,\cdots\right),
$$ for all $(x_n)\in \ell^2(\N)$.
Here $m(T)=0=\|Te_1\|$. Thus $T\in \ma(\ell^2(\mathbb{N}))$. Consider the subspace $M=\bigvee\{e_n : n\geq 2\}$, where $\bigvee$ denotes the closed linear span. Then $M\in \red_T$. Now $m(T_M)=\frac{1}{2}$ but $\frac{1}{2}\notin \sigma_p(T_M)$, and so by Proposition \ref{eigenvalue}, $T\notin \ma_r(\ell^2(\N))$. Therefore, $\ma_r(\ell^2(\N))\subsetneq \ma(\ell^2(\N))$.
\end{example} 

It is to be noted that $\ma(H)$ is dense in $\clb(H)$ with respect to the operator norm of $\clb(H)$ \cite{SHK:RG2018}. Although $\ma_r(H)$ is a proper subclass of $\ma(H)$, we later show that $\overline{\ma_r(H)}^{\|\cdot\|_{op}}=\clb(H)$. This improves the density result.

\begin{lemma}\label{m(T) for direct sum operators}
    Let $T_n\in \clb(H_n)$ for $n\in\N$, and consider the operator $T=\bigoplus_{n\in\N}T_n\in\clb\left(\bigoplus_{n\in\N}H_n\right)$. Then the minimum modulus of $T$ is given by $m(T)=\inf_{n\in\N}m(T_n)$.
\end{lemma}

\begin{proof}
    Let $x=(x_n)\in \bigoplus_{n\in\N}H_n$ be any unit vector, where $x_n\in H_n$ for $n\in\N$. Then by the definition \eqref{direct sum operators}, we have $$\|Tx\|^2=\sum_{n\in\N}\|T_nx_n\|^2 \geq \sum_{n\in\N}m(T_n)^2\|x_n\|^2 \geq \inf_{n\in\N}m(T_n)^2,$$ which further implies $m(T)\geq \inf_{n\in\N} m(T_n)$.

    On the other hand, observe that $$m(T)\leq \inf_{x_n\in H_n,\,\|x_n\|=1}\|T(0,\cdots,x_n,0,\cdots)\|=\inf_{x_n\in S_{H_n}}\|T_n x_n\|=m(T_n)$$ for all $n\in\N$.
    Therefore, $m(T)\leq \inf_{n\in\N} m(T_n)$. Combining this with the previous inequality, we obtain the desired equality.
\end{proof}
\begin{lemma}\label{ma set}
Let $T \in \clb(H)$ and $G=\{ x \in H : \|Tx\| = m(T)\|x\| \}$. Then $G=N(T^*T- m(T)^2 I)$.
Moreover,
$$
G =
\begin{cases}
N(T), & \text{if } m(T)=0, \\
N(|T| - m(T)I), & \text{if } m(T)>0.
\end{cases}
$$
\end{lemma}

\begin{proof}
    Let $x\in G$. Then $\|Tx\|^2=m(T)^2\|x\|^2$, that is, $\langle T^*Tx,x\rangle = m(T)^2 \langle x,x\rangle$,
    which is equivalent to
    $$\langle (T^*T - m(T)^2 I)x, x\rangle = 0.$$
    Since $T^*T - m(T)^2 I \ge 0$, it follows that $(T^*T - m(T)^2 I)x = 0.$ Thus $x\in N(T^*T-m(T)^2I)$.

    Conversely, let $x\in N(T^*T-m(T)^2I)$. Then $T^*Tx=m(T)^2x$, and hence $\|Tx\|^2=m(T)^2\|x\|^2.$ Thus $x\in G$. Therefore $G=N(T^*T-m(T)^2I)$.
    
\smallskip
    If $m(T)=0$, then $G=N(T^*T) = N(T)$. Next, assume that $m(T)>0$. Then the operator $|T|+m(T)I$ is invertible. Since $T^*T - m(T)^2 I=(|T| + m(T)I)(|T| - m(T)I)$, we get $N(T^*T - m(T)^2)=N(|T| - m(T)I)$. In this case, $G=N(|T| - m(T)I)$.
\end{proof}

\begin{remark}
   Combining above two cases, we can say $G=N(|T|-m(T)I)$. For $T\in \clb(H)$, 
   \begin{enumerate}
       \item $T\in\ma(H)$ if and only if $G\neq \{0\}$.

    \item If $T$ is normal, then both $T$ and $T^*$ commute with $|T|$. Consequently, $G$ reduces $T$.
    \end{enumerate}  
\end{remark}
\medskip

Let $T\in\clb(H)$ and $M\in \red_T$. Then there exists an orthogonal projection $P$ of $H$ onto $M$ such that $PT=TP$ (see \cite[Proposition 3.7, Page 39]{Conway}), and equivalently $T^*P=PT^*$. It is easy to observe that $PT^*T=T^*TP$. Thus $M$ reduces $T^*T$. Therefore, we have $\red_T\subseteq \red_{T^*T}$. With the help of continuous functional calculus, it can be shown that $\red_{|T|}=\red_{T^*T}$.

\begin{proposition}\label{equivalence}
    Let $T\in\clb(H)$. Then we have the following:
    \begin{enumerate}
        \item $|T|\in \ma_r(H)_+$ if and only if $T^*T\in \ma_r(H)_+$
        
        \item If $|T|\in\ma_r(H)_+$, then $T\in \ma_r(H)$.

        \item If $T\in\clb(H)_{+}$, then $T\in\ma_r(H)$ if and only if $T^2\in\ma_r(H)$.
    \end{enumerate}
\end{proposition}
\begin{proof}
    \textbf{Proof of (i)}: For a nonzero $M\in\red_{|T|}=\red_{T^*T}$, observe that 
    \begin{align*}
        m(|T|_M)^2=\inf_{x\in S_M} \||T|x\|^2=\inf_{x\in S_M} \langle T^*Tx,x\rangle = m(T^*T|_M),
    \end{align*}
    where the last equality follows from \cite[Proposition 2.2]{Carvajal 2}. Now $|T|_M\in\ma(M)$ if and only if $m(|T|_M)\in\sigma_p(|T|_M)$, that is, there exists a nonzero $x\in M$ such that $|T|_M x=m(|T|_M)x$. Equivalently, $T^*T|_M x= m(|T|_M)^2 x=m(T^*T|_M)x$, and hence $T^*T|_M\in\ma(M)$. Since $M$ is arbitrary, the statement follows.
    
    \smallskip
    \textbf{Proof of (ii)}: Let $|T|\in \ma_r(H)$. Then $|T|_M\in \ma(M)$ for every nonzero $M\in \red_{|T|}$. Since $\red_T\subseteq \red_{|T|}$, we get $|T|_N\in \ma(N)$ for every nonzero $N\in\red_T$. But $\|Tx\|=\||T|x\|$ for all $x\in N$. Thus $T_N\in\ma(N)$ for every nonzero $N\in \red_T$, and so $T\in\ma_r(H)$.

    \smallskip
    \textbf{Proof of (iii):} This is a direct consequence of (i).  
\end{proof}

\begin{remark}
     In general, $T\in\ma_r(H)$ neither implies $T^*\in\ma_r(H)$ nor $|T|\in\ma_r(H)_+$, and does not guarantee that $R(T)$ is closed. The following example illustrates these phenomena. However, we shall show later that these properties hold when $T$ is normal.
\end{remark}

\begin{example}
Consider the operator $T: \ell^2(\N)\longrightarrow \ell^2(\N)$ such that $T(e_1)=0$ and $T(e_n)=\frac{1}{2^n}e_{n-1}$ for $n\geq 2$. Here $T$ is irreducible, that is, $\ell^2(\N)$ is the only nonzero reducing subspace for $T$. Also $\|T(e_1)\|=0=m(T)$, and so $T\in\ma_r(\ell^2(\N))$. 

It can be shown that $T^*(e_n)=\frac{1}{2^{n+1}}e_{n+1}$ for all $n\in\N$, and $m(T^*)=0$. Since $T^*$ is injective, there does not exist any $x\in S_{\ell^2(\N)}$ such that $\|T^*x\|=0$. Therefore, $T^*\notin \ma_r(\ell^2(\N))$.

Observe that $T^*T(e_1)=0$ and $T^*T(e_n)=\frac{1}{4^n}e_n$ for all $n\geq 2$. Consider the subspace $M=\bigvee\{e_n:n\geq 5\}$, which reduces $T^*T$. But $m(T^*T|_M)=0\notin \sigma_p(T^*T|_M)$, and hence $T^*T\notin\ma_r(\ell^2(\N))$. By Proposition \ref{equivalence}, it follows that $|T|\notin\ma_r(\ell^2(\N))$.

Note that $N(T)^{\perp}=\bigvee\{e_n : n\geq 2\}$. Since $m(T|_{N(T)^{\perp}})=0$ and $T|_{N(T)^\perp}$ is injective, $R(T)$ is not closed.

We remark that the operator $\widetilde{T}=T\oplus 2I$ on $\ell^2(\N)\oplus \ell^2(\N)$ being reducible, serves a nontrivial example of the above phenomena. In this case, $\red_{\widetilde{T}}=\bigg\{\{0\}\oplus N, \ell^2(\N)\oplus N : N \text{ is any closed subspace of } \ell^2(\N)\bigg\}$.
\end{example}

An operator $T\in \clb(H)$ is said to be $WEP$ (Weak Equal Projection), if $R(T)=R(T^*)$. This condition implies that $N(T)=N(T^*)$. In addition, if $R(T)$ is closed, then we say $T$ to be an $EP$ operator. The class 
of weak EP operators includes normal operators and invertible operators; see \cite{EP2024}.

\begin{proposition}\label{WEP T and T* eqv}
Let $T$ be a $WEP$-operator on $H$. Then $T\in\ma_r(H)$ if and only if
$T^*\in\ma_r(H)$.
\end{proposition}

\begin{proof}
Let $M$ be a nonzero reducing subspace for $T$. Then it also reduces $T^*$, and
$N(T_M)=M\cap N(T)$, $N(T^*_M)=M\cap N(T^*)$. For a $WEP$ operator these
kernels coincide. By \cite[Corollary 3.15]{NBGRAM}, we get that $T_M\in\ma(M)$ if and only if $T^*_M\in\ma(M)$. Since $M$ is arbitrary, the desired equivalence follows.
\end{proof}

We now show that every idempotent operator belongs to $\ma_r(H)$.

\begin{proposition}\label{Idempotent}
    Let $T$ be a nonzero idempotent operator on $H$. Then $T\in \ma_r(H)$.
\end{proposition}

\begin{proof}
First, we show that $T$ is minimum attaining.
\begin{enumerate}
    \item  If $N(T) \neq \{0\}$, choose a unit vector $x \in N(T)$. Then $\|Tx\| = 0 = m(T)$, which implies $T \in \ma(H)$.
    
    \item If $N(T) = \{0\}$, then $T$ is injective. Thus $T^2=T$ gives $T(T - I) = 0$, and so $T-I=0$. Therefore, $T=I$, and hence $T\in\ma(H)$.
    \end{enumerate}

    Let $M\in\red_T$. Then $T_M$ is also an idempotent operator on $M$. Since $T_M\in \ma(M)$ for all nonzero $M\in\red_T$, we have $T\in \ma_r(H)$.
\end{proof}

\begin{remark}
  Let $T$ be an idempotent operator such that both $R(T)$ and $N(T)$ are infinite dimensional. Then by Proposition \ref{Idempotent}, $T\in\ma_r(H)$. On the other hand, by \cite[Theorem 6.8]{Block AN 2026}, $T$ is not $\am$. Therefore, $\am(H)\subsetneq\ma_r(H)$.
\end{remark}

\begin{corollary}\label{Ortho proj}
    Let $P$ be an orthogonal projection on $H$. Then $P\in \ma_r(H)$.
\end{corollary}

\begin{corollary}\label{partial isometry}
    Let $V$ be a partial isometry on $H$. Then $V\in \ma_r(H)$.
\end{corollary}

\begin{proof}
    As $V$ is a partial isometry, $V^*V$ is the orthogonal projection onto $R(V^*)$. By Corollary \ref{Ortho proj}, we get $V^*V\in \ma_r(H)$. Consequently, $V\in \ma_r(H)$, by Proposition \ref{equivalence}.
\end{proof}

Let $\T=\{z\in \C : |z|=1\}$ be the unit circle in the complex plane $\C$, and let $L^2(\T)$ denote the space of square-integrable functions on $\T$ with respect to the normalized Lebesgue measure $\mu$. We denote the Hardy Hilbert space by $H^2$, which is defined as $$H^2 :=\{f\in L^2(\T) : \widehat{f}(n)=0 \quad \forall n < 0 \},$$ where $\widehat{f}(n)=\langle f, \chi_{n} \rangle$ is the $n^{th}$ Fourier coefficient of $f$ and $\chi_{n} (z)=z^n$ for $z\in \T$, $n\in \mathbb{Z}$. Let $L^\infty (\T)$ denote the Banach space of essentially bounded measurable functions on $\T$. The Hardy space $H^\infty$ is defined as
$$H^\infty := H^2\cap L^\infty(\T).$$
A function $\vp\in H^\infty$ is called \emph{inner} if $|\vp(z)|=1$ for a.e. $z\in\T$.

For $\vp\in L^\infty(\T)$, the Toeplitz operator $T_\vp$ and the Hankel operator $H_\vp$ on $H^2$ are defined, respectively, by
$$
T_\vp f=P(\vp f), \qquad f\in H^2,
$$
and
$$
H_\vp f=J(I-P)(\vp f), \qquad f\in H^2,
$$
where $P:L^2(\T)\to L^2(\T)$ denotes the orthogonal projection onto $H^2$,
and $J:L^2(\T)\to L^2(\T)$ is the unitary operator defined by
$$
J(z^n)=z^{-n-1},\qquad n\in\mathbb{Z}.
$$
 For basic properties of Hardy spaces, Toeplitz operators and Hankel operators, we refer to \cite{Douglas,Rosenthal}.
\begin{example}
    For an inner function $\vp\in H^\infty$, the Toeplitz operator
$T_\vp$ is a partial isometry by \cite{PIT}. Since every partial isometry on $H^2$ belongs to $\ma_r(H^2)$, it follows that $T_\vp\in\ma_r(H^2)$.
\end{example}

\begin{example}
    Let $\overline{\vp}\in H^\infty$ be an inner function. Then by \cite[Exercise 4.19, Page 159]{Rosenthal}, the Hankel operator $H_{z\vp}$ is a partial isometry on $H^2$, and thus $H_{z\vp}\in\ma_r(H^2)$.
\end{example}

\begin{proposition}
    Let $T,S\in \clb(H)$. If there exists a unitary operator $U$ on $H$ such that $T=U^*SU$. Then $T\in \ma_r(H)$ if and only if $S\in \ma_r(H)$.
\end{proposition}

\begin{proof}
    Let $S\in \ma_r(H)$ and $M\in \red_T$.  Then $T(M)\subseteq M$ and $T^*(M)\subseteq M$, which further implies $SU(M)\subseteq U(M)$ and $S^*U(M)\subseteq U(M)$. Thus $U(M)\in \red_{S}$. Since $S\in \ma_r(H)$, $S_{U(M)}$ is minimum attaining. Consequently, there exists a unit vector $y\in U(M)$ such that $\|S_{U(M)}y\|=m(S_{U(M)})$. Write $y=Ux$ with $x\in M$. Since $U$ is an isometry, $\|x\|=1$. Further, since $U^*$ is an isometry, $$\|U^*SUx\|=\|SUx\|=\|S_{U(M)}y\|=m(S_{U(M)})=m\bigl(U^*SU|_M\bigr).$$ 
    Hence $\|T_M x\|=m(T_M)$, and therefore $T\in \ma_r(H)$.
    
\smallskip
    Conversely, assume that $T\in \ma_r(H)$. Then by the above argument, $S=(U^*)^*\,TU^*\in \ma_r(H)$.
\end{proof}

\section{Spectral Structure of Operators in $\ma_r(H)$}

In this section, we first show that every self-adjoint operator in $\ma_r(H)$ is diagonalizable, and the absolute values of its eigenvalues form a countable well-ordered set. This leads to a spectral representation in terms of mutually orthogonal projections. We then specialize to positive operators and obtain a complete characterization of those belonging to $\ma_r(H)$ in terms of their spectral representations. We also establish a connection between $\ma_r(H)$ and the class $\beta(H)$ through an affine transformation of positive operators.

\begin{theorem}\label{diagonalizable}
    Let $T\in \ma_r(H)$ be self-adjoint. Then $H$ admits an orthonormal basis consisting of eigenvectors of $T$.
\end{theorem}

\begin{proof}
    Let $\mathcal{L}=\{E : E \text{ is an orthonormal set of eigenvectors of } T\}$. Since $T\in\ma_r(H)$ is self-adjoint and $H\in \red_T$, by Proposition \ref{eigenvalue}, either $m(T)$ or $-m(T)$ is an eigenvalue of $T$. Hence we have $\mathcal{L}\neq \emptyset$. 

    Let $\mathcal{C}\subseteq \mathcal{L}$ be a chain with respect to the set inclusion. Define, 
    $E_0 := \bigcup_{E \in \mathcal{C}} E$.
    We show that $E_0\in \mathcal{L}$. Clearly, every element of $E_0$ is an eigenvector of $T$. Let $x,y \in E_0$. Then there exist $E_1, E_2 \in \mathcal{C}$ such that $x \in E_1$ and $y \in E_2$. Since $\mathcal{C}$ is totally ordered by inclusion, either $E_1 \subseteq E_2$ or $E_2 \subseteq E_1$. Without loss of generality, assume that $E_1 \subseteq E_2$. Then $x,y \in E_2$, and since $E_2$ is orthonormal, we obtain
$$
\langle x,y \rangle = 0 \quad \text{whenever } x \neq y, \text{ and } \|x\|=\|y\|=1.
$$
Clearly, $E_0$ is orthonormal, and hence $E_0 \in \mathcal{L}$. Therefore, every chain in $\mathcal{L}$ admits an upper bound in $\mathcal{L}$. By Zorn's lemma, $\mathcal{L}$ has a maximal element, say $\cle$.

\smallskip
Let $\widetilde{\cle}:=\bigvee \cle$, the closed linear span of $\cle$.   
We show that $H=\widetilde{\cle}$, equivalently $\widetilde{\cle}^\perp=\{0\}$. Since every
$e\in\cle$ is an eigenvector of $T$, we have $Te\in span\{e\}\subseteq span(\cle)$. By the linearity and continuity of $T$, it follows that $T(\widetilde{\cle})\subseteq \widetilde{\cle}$. Thus $\widetilde{\cle}$ is invariant for $T$. Since $T$ is self-adjoint, $\widetilde{\cle}^\perp$ is also invariant for $T$. Consequently, $\widetilde{\cle}^\perp\in\red_T$. 

Assume that $\widetilde{\cle}^\perp\neq \{0\}$. Then we have
$T_{\widetilde{\cle}^\perp}\in \ma(\widetilde{\cle}^\perp)$, and so by Proposition \ref{eigenvalue}, either $m(T_{\widetilde{\cle}^\perp})$ or $- m(T_{\widetilde{\cle}^\perp})$ is an eigenvalue of $T_{\widetilde{\cle}^\perp}$. In either case, there exists a unit eigenvector of $T$ in $\widetilde{\cle}^\perp$, which contradicts the maximality of $\cle$. Hence $\widetilde{\cle}^\perp=\{0\}$. Therefore, $\cle$ forms an orthonormal basis of $H$. 
\end{proof}

\begin{corollary}\label{positive diagonalizable}
    Let $T\in \ma_r(H)$ be positive. Then $H$ admits an orthonormal basis consisting of eigenvectors of $T$.
\end{corollary}

\begin{proposition}\label{non decreasing}
Let $T\in \clb(H)$ be a positive operator.
Assume that $T$ has a strictly decreasing sequence of eigenvalues. Then $T\notin \ma_r(H)$.
\end{proposition}

\begin{proof}
Let $(\lambda_n)$ be a strictly decreasing sequence of eigenvalues of
$T$. Since $(\lambda_n)$ is bounded below by $0$, there exists
$\lambda\geq 0$ such that $\lambda_n\searrow \lambda$. Let $M_n = N(T-\lambda_n I)$ for $n\in\N$. Since $T$ is positive, the eigenspaces $(M_n)_{n\in\mathbb N}$
corresponding to distinct eigenvalues are mutually orthogonal.

\smallskip
Define $M:=\bigoplus_{n=1}^{\infty} M_n$. Since $T$ is positive and each $M_n\in \red_T$, we get $M\in \red_T$.
Now by Lemma \ref{m(T) for direct sum operators}, $m(T_M)=\inf_{n\in \mathbb{N}} \lambda_n$. Since $\sigma_p(T_M)\subseteq \sigma_p(T)$, if $m(T_M)\in\sigma_p(T_M)$, then $m(T_M)=\lambda_k$ for some $k\in\N$. But $(\lambda_n)$ being strictly decreasing implies $\lambda_{k+1}<\lambda_k$, which contradicts that $m(T_M)=\inf_{n\in\N}\lambda_n$. Therefore, $m(T_M)\notin \sigma_p(T_M)$.
  
Now for any $(x_n)\in S_M$,
$$
\|T_M(x_n)\|^2=\sum_{n=1}^\infty \lambda_n^2\|x_n\|^2 > m(T_M)^2 \sum_{n=1}^\infty \|x_n\|^2 = m(T_M)^2.
$$
Hence $T_M\notin \ma(M)$, and so $T\notin \ma_r(H)$.
\end{proof}

\begin{theorem}\label{self-adjoint}
Let $T\in\ma_r(H)$ be self-adjoint, and let
$$
\Gamma:=\{|\lambda|:\lambda\in\sigma_p(T)\}\subseteq[0,\|T\|].
$$
Then the following hold:
\begin{enumerate}
\item $T$ is diagonalizable;

\item $\Gamma$ is a countable well-ordered subset of $[0,\|T\|]$;

\item If $P_\lambda$ denotes the orthogonal projection onto
      $N(T-\lambda I)$ for each $\lambda\in\sigma_p(T)$, then
      \begin{equation}
T=\sum_{\lambda\in\sigma_p(T)}\lambda P_\lambda\quad\text{and}\quad
|T|=\sum_{\lambda\in\sigma_p(T)}|\lambda|P_\lambda,
\end{equation}
where both series converge in the strong operator topology.
\end{enumerate}
\end{theorem}

\begin{proof}
\textbf{Proof of (i):} By Theorem \ref{diagonalizable}, it follows that $T$ is diagonalizable. 

\smallskip
\textbf{Proof of (ii):} For each $\lambda\in\sigma_p(T)$, choose a unit vector $x_\lambda\in M_\lambda:=N(T-\lambda I)$. Then $\{x_\lambda:\lambda\in\sigma_p(T)\}$
is an orthonormal set in $H$. Since $H$ is separable, every orthonormal set is countable. Consequently, $\sigma_p(T)$ is countable, and hence $\Gamma$ is countable.

\smallskip
It remains to prove that $\Gamma$ is a well-ordered set. Let $A$ be a nonempty subset of $\Gamma$, and define
$$
M_A:=\bigoplus_{|\lambda|\in A} M_\lambda.
$$
Since $M_A$ is an orthogonal sum of eigenspaces of the self-adjoint
operator $T$, it is a nonzero reducing subspace for $T$. Thus, by the assumption $T\in\ma_r(H)$, the restriction $T_{M_A}$ is minimum attaining. 

For any $x\in M_A$, write $x=\sum_{|\lambda|\in A}x_\lambda$, where
$x_\lambda\in M_\lambda$. Then we have
$$
\|Tx\|^2
=\sum_{|\lambda|\in A}|\lambda|^2\|x_\lambda\|^2
\ge
(\inf A)^2\|x\|^2.
$$
Thus $m(T_{M_A})\ge\inf A$. On the other hand, for every $\varepsilon>0$, there exists $a\in A$ such that $a<\inf A+\varepsilon$. Choose $\lambda_0\in\sigma_p(T)$ with $|\lambda_0|=a$, and let
$u\in M_{\lambda_0}$ be a unit vector. Then $u\in M_A$ and
$$
m(T_{M_A})\le\|Tu\|=|\lambda_0|=a
<\inf A+\varepsilon.
$$
Since $\varepsilon>0$ is arbitrary, it follows that $m(T_{M_A})\leq\inf A$. Therefore, we obtain $m(T_{M_A})=\inf A$.

Since $T_{M_A}$ is minimum attaining, there exists $x_0\in M_A$ with $\|x_0\|=1$ such that
    $$
        \|Tx_0\|=m(T_{M_A})=\inf A.
    $$
Writing $x_0=\sum_{|\lambda|\in A}x^{'}_\lambda$, we have
$$
0=\|Tx_0\|^2-(\inf A)^2\|x_0\|^2
=\sum_{|\lambda|\in A}
\bigl(|\lambda|^2-(\inf A)^2\bigr)\|x^{'}_\lambda\|^2.
$$
Since each summand is nonnegative and $x_0\neq0$, there exists
$\lambda\in\sigma_p(T)$ with $x^{'}_\lambda\neq0$. For such a $\lambda$ we must have $|\lambda|^2=(\inf A)^2$. As $A\subseteq[0,\|T\|]$, it follows that $|\lambda|=\inf A$, and hence $\inf A\in A$.

Therefore, every nonempty subset of $\Gamma$ has a least element.
Consequently, $\Gamma$ is well-ordered.
\medskip

\textbf{Proof of (iii):} Since $T$ is diagonalizable, we have
$$
H=\bigoplus_{\lambda\in\sigma_p(T)} M_\lambda,\text{ where } M_\lambda:=N(T-\lambda I).
$$
Since $T$ is self-adjoint, eigenspaces corresponding to distinct eigenvalues are mutually orthogonal. Therefore, if $P_\lambda$ is the orthogonal projection onto $M_\lambda$, then $\sum_{\lambda\in\sigma_p(T)}P_\lambda=I$ in the strong operator topology. Thus for every $x\in H$, we obtain $$Tx=\sum_{\lambda\in\sigma_p(T)}\lambda P_\lambda x \text{ and } |T|x=\sum_{\lambda\in\sigma_p(T)}|\lambda|P_\lambda x,$$ where both series converge in the strong operator topology. 
\end{proof}

\begin{theorem}\label{Red Min Necessary}
    Let $T\in\ma_r(H)$ be positive. Then the following hold: 
    \begin{enumerate}
        \item The set $\Lambda:=\sigma_p(T)$ is a countable well-ordered subset of $[0,\|T\|]$;

        \item There exist mutually orthogonal projections $(P_\lambda)_{\lambda\in\Lambda}$ on $H$ with $\sum_{\lambda\in\Lambda}P_\lambda=I$  such that 
        \begin{equation}T=\sum_{\lambda\in\Lambda}\lambda P_\lambda,
        \end{equation}
        where the series converges in the strong operator topology.
    \end{enumerate}
\end{theorem}

\begin{proof}
    Since $T$ is positive, we have $\sigma_p(T)\subseteq [0,\|T\|]$. Hence
    $$\Gamma:=\{|\lambda|:\lambda\in\sigma_p(T)\}
=\sigma_p(T)=\Lambda.$$
The conclusions follow immediately from Theorem \ref{self-adjoint}.
\end{proof}

\begin{theorem}\label{Red min sufficient}
Let $\Lambda\subseteq[0,\infty)$ be a bounded countable well-ordered set, and
let $(P_\lambda)_{\lambda\in\Lambda}$ be mutually orthogonal projections with $\sum_{\lambda\in\Lambda}P_\lambda=I$
in the strong operator topology. Then the operator
$$
T:=\sum_{\lambda\in\Lambda}\lambda P_\lambda\in\ma_r(H)_+.
$$
\end{theorem}

\begin{proof}
Clearly, $T\ge0$. Let $M\ne\{0\}$ reduce $T$, and let $P_M$ be the
orthogonal projection onto $M$. Since $P_MT=TP_M$, the spectral theorem gives
$P_MP_\lambda=P_\lambda P_M$ for every $\lambda\in\Lambda$. Thus
$P_\lambda M\subseteq M$.

Consider the set
$\Lambda_M:=\{\lambda\in\Lambda:P_\lambda M\ne\{0\}\}$.
This set is nonempty, and because $\Lambda$ is well-ordered, it has a least
element $\lambda_0$. Choose a unit vector $x_0\in P_{\lambda_0}M$. Then
$Tx_0=\lambda_0x_0$. Also for every unit vector $x\in M$, we have
$$
\|Tx\|^2=\sum_{\lambda\in\Lambda_M}\lambda^2\|P_\lambda x\|^2
\ge\lambda_0^2.
$$
Hence $m(T_M)=\lambda_0=\|Tx_0\|$. Therefore, $T_M$ is minimum attaining.
Since $M\in \red_T$ is arbitrary, $T\in\ma_r(H)_+$.
\end{proof}
\begin{theorem}\label{+M_r class equiv}
Let $T\in\clb(H)$ be positive. Then $T\in\ma_r(H)_+$ if and only if there exist a
countable well-ordered set $\Lambda\subseteq[0,\|T\|]$ and a family of mutually
orthogonal nonzero projections $(P_\lambda)_{\lambda\in\Lambda}$ such that
$$
\sum_{\lambda\in\Lambda}P_\lambda=I\, \text{ and }\,
T=\sum_{\lambda\in\Lambda}\lambda P_\lambda,
$$
where the above two series converge in the strong operator topology. Furthermore,
$\sigma_p(T)=\Lambda$ and $\sigma(T)=\overline{\Lambda}$.
\end{theorem}

\begin{proof}
The necessity is Theorem \ref{Red Min Necessary}, and the sufficiency is
Theorem \ref{Red min sufficient}.

Since each $P_\lambda\neq 0$ and
$TP_\lambda=\lambda P_\lambda$, we have $\Lambda\subseteq\sigma_p(T)$. Conversely, let $\mu\in\sigma_p(T)$, that is,
$Tx=\mu x$ for some nonzero $x\in H$. Since $P_\lambda T=TP_\lambda=\lambda P_\lambda$, we obtain $(\lambda-\mu)P_\lambda x=0$ for every $\lambda\in\Lambda$.
If $\mu\notin\Lambda$, then $P_\lambda x=0$ for all $\lambda\in\Lambda$. Since
$\sum_{\lambda\in\Lambda}P_\lambda=I$ in the strong operator
topology, it follows that $x=0$, which is a contradiction. 
Hence $\sigma_p(T)=\Lambda$, and therefore $\overline{\Lambda}\subseteq\sigma(T)$.

If $\mu\notin\overline{\Lambda}$, then $d:=\operatorname{dist}(\mu,\Lambda)>0$ and
$\sum_{\lambda\in\Lambda}(\lambda-\mu)^{-1}P_\lambda$ is a bounded operator of
norm at most $1/d$ that inverts $T-\mu I$. Consequently, $\mu\notin\sigma(T)$. Thus we get $\sigma(T)\subseteq\overline{\Lambda}$.
\end{proof}

\begin{corollary}\label{spectrum positive alpha}
Let $T\in\ma_r(H)_+$. Then $\sigma_p(T)$ is countable and well-ordered. Moreover,
$\sigma(T)=\overline{\sigma_p(T)}$.
\end{corollary}

\begin{proof}
This is immediate from Theorem \ref{+M_r class equiv}.
\end{proof}

\begin{remark}
Here, the multiplicities of eigenvalues are not taken into account. Note that an eigenvalue may have infinite multiplicity in some cases. For example, the
identity operator on an infinite dimensional Hilbert space has the eigenvalue
$1$ with infinite multiplicity.
\end{remark}

\begin{corollary}\label{self mod T}
Let $T\in\ma_r(H)$ be self-adjoint. Then $|T|\in\ma_r(H)_+$.
\end{corollary}

\begin{proof}
By Theorem \ref{self-adjoint}, we have $|T|$ is diagonalizable and its distinct
eigenvalues form a bounded countable well-ordered subset of $[0,\infty)$. Hence by
Theorem \ref{Red min sufficient}, we obtain $|T|\in\ma_r(H)_+$.
\end{proof}

\begin{remark}\label{self-adj iff}
For $T\in\clb(H)_{sa}$, $T\in\ma_r(H)$ if and only if
$|T|\in\ma_r(H)_+$. The reverse implication follows from
Proposition \ref{equivalence}(ii).
\end{remark}

Recall that an operator $T\in\clb(H)$ belongs to $\beta(H)$ if for every nonzero $M\in\red_T$, the restriction $T_M$ is norm attaining. We next establish a relation between the classes $\ma_r(H)$ and $\beta(H)$.

\begin{proposition}\label{alpha and beta connection}
Let $T\in\clb(H)_+$ and $\|T\|\le\delta$. Then
$$
T\in\ma_r(H)\quad\Longleftrightarrow\quad \delta I-T\in\beta(H).
$$
\end{proposition}

\begin{proof}
Let $S:=\delta I-T$. Since $S$ is an affine function of $T$, a closed subspace
reduces $S$ if and only if it reduces $T$. Let $M\ne\{0\}$ be such a reducing
subspace. Since $0\le T_M\le\delta I_M$, we have
$$
\|S_M\|=\delta-m(T_M).
$$
If $T_M$ attains its minimum at a unit vector $x\in M$, then $T_M x=m(T_M)x$, that is,
$\|Sx\|=\delta-m(T_M)=\|S_M\|$. Therefore, $S_M$ attains its norm.

\smallskip
Conversely, if $S_M$ attains its norm at a unit vector $x\in M$, positivity of $S_M$ implies that
$$
Sx=\|S_M\|x=(\delta-m(T_M))x,
$$
and hence $Tx=m(T_M)x$. Thus $T_M$ is minimum attaining if and only if
$S_M$ is norm attaining. The conclusion follows for every reducing $M$.
\end{proof}

Next, we provide a connection between the classes $\mathcal{M}_r(H)$
and $\beta(H)$ for invertible operators.
\begin{proposition}\label{inverse beta}
Let $T\in\clb(H)$ be invertible. Then
\begin{enumerate}
    \item $\mathcal R_T=\mathcal R_{T^{-1}}$;
    
    \item $T\in\ma_r(H)$ if and only if $T^{-1}\in\beta(H)$.
\end{enumerate}
\end{proposition}

\begin{proof}
\textbf{Proof of (i):} If $P$ is an orthogonal projection with $PT=TP$, then $PT^{-1}=T^{-1}P$, and
conversely. Hence $\mathcal R_T=\mathcal R_{T^{-1}}$. 

\smallskip
\textbf{Proof of (ii):}
Let $M\in\mathcal R_T$ be
nonzero. Then $T^{-1}M\subseteq M$, so for $x\in M$ we have $x=T(T^{-1}x)\in TM$.
Thus $T_M$ is invertible on $M$ with $(T_M)^{-1}=(T^{-1})_M$, and consequently
\begin{equation}\label{inv mod}
m(T_M)=\|(T^{-1})_M\|^{-1}.
\end{equation}
Suppose $T_M\in\mathcal M(M)$, say $\|Tx\|=m(T_M)$ for a unit vector $x\in M$, and let
$y=Tx\in M$. Then $\|T^{-1}y\|=1$ and $\|y\|=m(T_M)$. Thus by \eqref{inv mod}, we have
$$
\frac{\|(T^{-1})_My\|}{\|y\|}=\frac1{m(T_M)}=\|(T^{-1})_M\|,
$$
that is,  $(T^{-1})_M$ attains its norm.

Conversely, let $(T^{-1})_M$ attains its norm at a unit vector $y\in M$, and consider $x=\dfrac{T^{-1}y}{\|T^{-1}y\|}$. Then by
\eqref{inv mod}, we obtain
$$
\|Tx\|=\frac{\|y\|}{\|T^{-1}y\|}=\frac1{\|(T^{-1})_M\|}=m(T_M).
$$
Since $M\in\mathcal R_T=\mathcal R_{T^{-1}}$ is arbitrary, the equivalence follows.
\end{proof}

\section{Stability and Density of $\ma_r(H)$}

In this section, we establish stability results for the class $\ma_r(H)$ under finite rank and compact perturbations. We further prove that every operator in $\clb(H)$ can be approximated in operator norm by operators in $\ma_r(H)$.

\begin{proposition}\label{finite rank perturbation}
Let $\Lambda\subseteq [0,\infty)$ be a bounded countable well-ordered set, and let
$(P_\lambda)_{\lambda\in\Lambda}$ be mutually orthogonal projections with $\sum_{\lambda\in\Lambda}P_\lambda=I$ in the strong operator topology. Suppose that
$F\in\mathcal F(H)_{sa}$ satisfies
$FP_\lambda=P_\lambda F$ for $\lambda\in\Lambda.$
Then
\begin{equation}
T=F+\sum_{\lambda\in\Lambda}\lambda P_\lambda
\end{equation}
belongs to $\mathcal M_r(H)$.
\end{proposition}

\begin{proof}
    Let $H_\lambda=R(P_\lambda)$ and $F_\lambda=F|_{H_\lambda}$ for $\lambda\in\Lambda$. Since $FP_\lambda=P_\lambda F$ for every $\lambda\in\Lambda$, each
$H_\lambda$ reduces $F$. Hence by the definition of $T$, each
$H_\lambda$ also reduces $T$. Thus we can write $T=\bigoplus_{\lambda\in\Lambda}
(\lambda I_{H_\lambda}+F_\lambda)$. Since $F$ has finite rank, $J:=\{\lambda\in\Lambda:F_\lambda\neq0\}$ is a finite set. Clearly, if $\lambda\notin J$, then $\lambda I_{H_\lambda}+F_\lambda=\lambda I_{H_\lambda}$.

Now let $\lambda\in J$. Since $F_\lambda$ is self-adjoint, there exists
an orthonormal basis of $H_\lambda$ consisting of eigenvectors of
$F_\lambda$. Therefore, the same basis consists of eigenvectors of
$\lambda I_{H_\lambda}+F_\lambda$. Hence $\lambda I_{H_\lambda}+F_\lambda$ is diagonalizable. It follows that every summand in the above orthogonal direct sum is
diagonalizable, and hence $T$ is diagonalizable. Moreover,
$$
\sigma_p(T)
=
(\Lambda\setminus J)
\cup
\bigcup_{\lambda\in J}
\sigma_p(\lambda I_{H_\lambda}+F_\lambda).
$$
Since $J$ is finite and each
$\sigma_p(\lambda I_{H_\lambda}+F_\lambda)$ is finite, $\sigma_p(T)$ is obtained from the well-ordered set $\Lambda$ by removing finitely many points and adjoining finitely many points. Hence $\sigma_p(T)$ is well-ordered countable set. Consequently, $\Gamma:=\{|\mu| : \mu\in\sigma_p(T)\}\subseteq[0,\|T\|]$ is a well-ordered countable set.

Let $Q_\mu$ denote the orthogonal projection onto
$N(T-\mu I)$ for each $\mu\in\sigma_p(T)$. Since $T$ is self-adjoint, the sequence $(Q_\mu)_{\mu\in\sigma_p(T)}$ is mutually orthogonal and $\sum_{\mu\in \sigma_p(T)}Q_\mu=I$. Then we obtain
$$
T=\sum_{\mu\in\sigma_p(T)}\mu Q_\mu \text{ and } |T|=\sum_{\mu\in\sigma_p(T)}|\mu| Q_\mu,
$$
where the series converges in the strong operator topology. By Theorem \ref{Red min sufficient}, we get $|T|\in\ma_r(H)_+$. The conclusion follows from Remark \ref{self-adj iff}.
\end{proof}

\begin{proposition}\label{compact perturbation lambda}
Let $\Lambda\subseteq [0,\infty)$ be a bounded countable well-ordered set and $(P_\lambda)_{\lambda\in\Lambda}$ be mutually orthogonal projections with $\sum_{\lambda\in\Lambda}P_\lambda=I$
in the strong operator topology. Suppose
$K\in\mathcal K(H)_+$ satisfies $KP_\lambda=P_\lambda K$ and $K|_{R(P_\lambda)}
\leq \lambda I_{R(P_\lambda)}$
for all $\lambda\in\Lambda$. Then the operator
\begin{equation}
T=\sum_{\lambda\in\Lambda}\lambda P_\lambda-K
\in\mathcal M_r(H)_+.
\end{equation}
\end{proposition}

\begin{proof}
Let $H_\lambda=R(P_\lambda)$ and 
$K_\lambda=K|_{H_\lambda}$ for $\lambda\in\Lambda$.
Since $KP_\lambda=P_\lambda K$, each $H_\lambda$ reduces $K$, and hence $T=
\bigoplus_{\lambda\in\Lambda}
(\lambda I_{H_\lambda}-K_\lambda)$.
Moreover, $0\leq K_\lambda\leq \lambda I_{H_\lambda}$ for each $\lambda\in\Lambda$ further implies that $T\geq0$.

Since $K_\lambda$ is positive and compact, $H_\lambda$ admits an
orthonormal basis consisting of eigenvectors of $K_\lambda$. Hence
$\lambda I_{H_\lambda}-K_\lambda$ is diagonalizable, and therefore $T$ is diagonalizable.

\smallskip
Next, we claim that $\Gamma:=\sigma_p(T)$
is well-ordered. Since $H$ is separable, $\Gamma$ is countable. For contradiction, assume that $\Gamma$ is not well
ordered. Then $\Gamma$ contains a strictly decreasing
sequence $(\mu_j)$. Let $T_\lambda:=\lambda I_{H_\lambda}-K_\lambda$. Since
$\sigma_p(T)=
\bigcup_{\lambda\in\Lambda}\sigma_p(T_\lambda)$,
for each $j$ choose $\lambda_j\in\Lambda$ such that $\mu_j\in\sigma_p(T_{\lambda_j})$.
Now choose a unit vector $x_j\in H_{\lambda_j}$ satisfying
$$
T_{\lambda_j}x_j=\mu_jx_j, \text{ and so } K_{\lambda_j}x_j=(\lambda_j-\mu_j)x_j.
$$
Since the sequence $(\lambda_j)$ is in the well-ordered set
$\Lambda$, it has either a constant subsequence or a strictly
increasing subsequence.

Suppose first that, after passing to a subsequence, $\lambda_j=\lambda$ for $j\in\N$.
Then we can write $K_\lambda x_j=(\lambda-\mu_j)x_j$. Since $(\mu_j)$ is strictly decreasing and $K_\lambda\geq 0$,
$$
0\leq\lambda-\mu_1
<
\lambda-\mu_2
<
\cdots.
$$
Thus $K_\lambda$ has infinitely many distinct nonzero eigenvalues forming
a strictly increasing sequence, which is impossible for a compact
operator.

Hence after passing to a subsequence, we may assume
$$
\lambda_1<\lambda_2<\cdots.
$$
Since $x_j\in H_{\lambda_j}$
and the subspaces $(H_{\lambda_j})_{\lambda_j\in\Lambda}$ are mutually
orthogonal, $(x_j)$ is an orthonormal sequence. Therefore, by compactness
of $K$, it follows that
$$
\|Kx_j\|\longrightarrow0, \text{ that is, } \lambda_j-\mu_j \longrightarrow0
$$

On the other hand, for every $j\geq 2$, we have
$$
\lambda_j-\mu_j
\geq \lambda_2-\mu_2
> \lambda_1-\mu_1
\geq 0.
$$
which contradicts the fact $\lambda_j-\mu_j\longrightarrow 0$.

Therefore, $\Gamma$ is well-ordered. By Theorem \ref{Red min sufficient}, we conclude that
$T\in\mathcal M_r(H)_+$.
\end{proof}

\begin{corollary}\label{corollary compact perturbation}
Let $\Lambda\subseteq[0,\infty)$ be a bounded countable well-ordered set, and
let $(P_\lambda)_{\lambda\in\Lambda}$ be mutually orthogonal projections
with $\sum_{\lambda\in\Lambda}P_\lambda=I$ in the strong operator topology. Suppose $K\in\clk(H)_+$ satisfies $KP_\lambda=P_\lambda K$ for $\lambda\in\Lambda$
and $\|K\|\leq \inf\Lambda$. Then the operator
\begin{equation}
\sum_{\lambda\in\Lambda}\lambda P_\lambda-K
\in\mathcal M_r(H)_+.
\end{equation}
\end{corollary}

\begin{proof}
For every $\lambda\in\Lambda$, observe that
$$
0\leq K|_{R(P_\lambda)}
\leq \|K\|I_{R(P_\lambda)}
\leq (\inf\Lambda)I_{R(P_\lambda)}
\leq \lambda I_{R(P_\lambda)}.
$$
Hence all the hypotheses of Proposition \ref{compact perturbation lambda}
are satisfied. Therefore,
$$
\sum_{\lambda\in\Lambda}\lambda P_\lambda-K
\in\mathcal M_r(H)_+.
$$
\end{proof}

\begin{theorem}\label{finite and compact perturbation suff}
Let $F\in\clf(H)_{sa}$, $K\in\mathcal K(H)_+$, and let
$(P_\lambda)_{\lambda\in\Lambda}$ be mutually orthogonal projections with $\sum_{\lambda\in\Lambda}P_\lambda=I$, $FP_\lambda=P_\lambda F$ and $KP_\lambda=P_\lambda K$ for all $\lambda\in\Lambda$,
where $\Lambda\subseteq[0,\infty)$ is a bounded countable well-ordered
set. Assume that $K|_{R(P_\lambda)}
\leq \lambda I_{R(P_\lambda)}$ for all $\lambda\in\Lambda$. If
$T:=\sum_{\lambda\in\Lambda}\lambda P_\lambda + F -K$ is positive, then $T\in\ma_r(H)_+$.
\end{theorem}

\begin{proof}
For $\lambda\in\Lambda$, write $H_\lambda:=R(P_\lambda)$ and $F_\lambda:=F|_{H_\lambda}$, $K_\lambda:=K|_{H_\lambda}$. Since $FP_\lambda=P_\lambda F$ and $KP_\lambda=P_\lambda K$, each
$H_\lambda$ reduces $F$, $K$, and hence $T$. Therefore, we can write
$$
T=\bigoplus_{\lambda\in\Lambda}T_\lambda,
\text{ where }
T_\lambda:=\lambda I_{H_\lambda}+F_\lambda-K_\lambda.
$$
Also it is clear that $0\leq K_\lambda\leq \lambda I_{H_\lambda}$.
Since $F$ has finite rank, there exists a finite set
$J\subseteq\Lambda$ such that $F_\lambda=0$ for $\lambda\in\Lambda\setminus J$. Let $\Lambda_0:=\Lambda\setminus J$ and $H_0:=\bigoplus_{\lambda\in\Lambda_0}H_\lambda$. Then we can write $$T_0:=T|_{H_0}
=\bigoplus_{\lambda\in\Lambda_0}
\left(\lambda I_{H_\lambda}-K_\lambda\right).$$
Clearly, the set $\Lambda_0$ is a bounded countable well-ordered subset of
$[0,\infty)$, and the projections $(P_\lambda)_{\lambda\in\Lambda_0}$
are mutually orthogonal with
$\sum_{\lambda\in\Lambda_0}P_\lambda=I_{H_0}$. Since
$K|_{H_0}\in\mathcal K(H_0)_+$ with
$K|_{H_0}P_\lambda=P_\lambda K|_{H_0}$ and $K|_{H_\lambda}\leq\lambda I_{H_\lambda}$ for $\lambda\in\Lambda_0$, by Proposition \ref{compact perturbation lambda}, we get $T_0\in\ma_r(H_0)_+$. In particular, by Theorem \ref{+M_r class equiv}, $T_0$ is diagonalizable and
$\sigma_p(T_0)$ is a countable well-ordered subset of $[0,\|T_0\|]$.

\smallskip
Now consider $\lambda\in J$. In this case we have
$$
T_\lambda
=\lambda I_{H_\lambda}+ F_\lambda-K_\lambda,
$$
where $F_\lambda-K_\lambda$ is compact and self-adjoint. Hence
$T_\lambda$ is diagonalizable. Since $F_\lambda$ is self-adjoint finite rank operator and $K_\lambda$ is positive compact operator, by \cite[Lemma 4.8]{Paulsen}, $C_\lambda:=F_\lambda-K_\lambda$ has at most finitely
many positive eigenvalues. If its negative eigenvalues are infinite,
they can be arranged as 
$$
\mu_1<\mu_2<\cdots<0,
\qquad \mu_j\longrightarrow0.
$$
Since each $T_\lambda$ is positive, we have $\lambda+\mu_j\geq0$ and $\lambda + \mu_j \nearrow \lambda$.

\begin{figure}[h]
\centering
\begin{tikzpicture}[xscale=1.25,yscale=1]

% Number line
\draw[->] (-5.0,0) -- (4.2,0);

% Negative eigenvalues accumulating at 0
\foreach \x in {
    -4.6,-4.1,-3.65,-3.25,-2.9,-2.58,-2.30,-2.05,
    -1.82,-1.62,-1.44,-1.28,-1.13,-1.00,-0.90,-0.82,
    -0.75,-0.69,-0.63,-0.58,-0.53,-0.48,-0.44,-0.40,
    -0.36,-0.33,-0.30,-0.27,-0.24,-0.22,-0.20,-0.18,
    -0.16,-0.14,-0.12,-0.10,-0.08,-0.06,-0.04,-0.02
}
{
    \draw (\x,0) -- (\x,0.42);
}

% Labels for selected eigenvalues
\node[below] at (-4.6,0) {$\mu_1$};
\node[below] at (-4.1,0) {$\mu_2$};
\node[below] at (-3.65,0) {$\mu_3$};
\node[below] at (-3.25,0) {$\mu_4$};
\node[below] at (-2.9,0) {$\mu_5$};
\node[below] at (-2.3,0) {$\cdots$};

% Zero
\draw[very thick] (0,0) -- (0,0.65);
\node[below] at (0,0) {$0$};

% Negative eigenvalue arrow
\draw[->] (-2.0,0.85) -- (-0.25,0.85);
\node[above] at (-1.15,0.85)
{$\mu_j\nearrow0$};

% Finitely many positive eigenvalues
\foreach \x in {0.65,1.25,2.0,2.85,3.55}
{
    \draw (\x,0) -- (\x,0.42);
}

% Positive eigenvalues label
\node[above] at (1.9,0.55)
{\scriptsize finitely many};

\end{tikzpicture}
\caption{The point spectrum of $C_\lambda=F_\lambda-K_\lambda$.}
\label{fig:spectrum-C}
\end{figure}
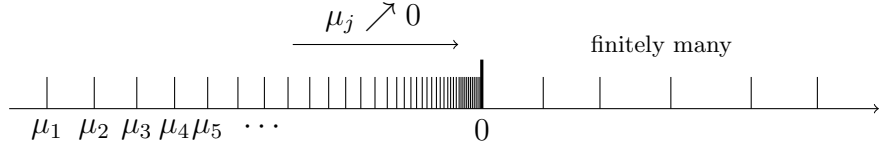

\begin{figure}[h]
\centering
\begin{tikzpicture}[xscale=1.25,yscale=1]

% Number line
\draw[->] (-4.8,0) -- (5.2,0);

% Eigenvalues below lambda
\foreach \x in {-4.5,-4.05,-3.45,-3.05,-2.7,-2.38,-2.10,
                -1.84,-1.60,-1.39,-1.20,-1.03,-0.88,-0.75,
                -0.63,-0.53,-0.44,-0.36,-0.29,-0.23,-0.18,
                -0.14,-0.11,-0.08,-0.055,-0.035,-0.025}
{
    \draw (\x,0) -- (\x,0.42);
}

% lambda
\draw[very thick] (0,0) -- (0,0.7);

% Eigenvalues above lambda
\foreach \x in {0.75,1.35,2.05,2.85,3.65}
{
    \draw (\x,0) -- (\x,0.42);
}

% Labels
\node[below] at (0,0) {$\lambda$};
\node[below] at (-4.5,0) {\fontsize{7pt}{8pt}\selectfont $0$};
\node[below] at (-4.05,0) {\fontsize{4.6pt}{5pt}\selectfont $\lambda+\mu_1$};
\node[below] at (-3.45,0) {\fontsize{4.6pt}{5pt}\selectfont $\lambda+\mu_2$};
\node[below] at (-3.05,0) {\fontsize{4.5pt}{5pt}\selectfont $\cdots$};
% Accumulation arrow
\draw[->] (-2.1,0.85) -- (-0.18,0.85);
\node[above] at (-1.15,0.85)
{$\lambda+\mu_j\nearrow\lambda$};

% Finitely many points above lambda
\node[above] at (2.2,0.55)
{\scriptsize finitely many};

\end{tikzpicture}
\caption{The point spectrum of
$T_\lambda=\lambda I_{H_\lambda}+C_\lambda$.}
\label{fig:spectrum-Tlambda}
\end{figure}
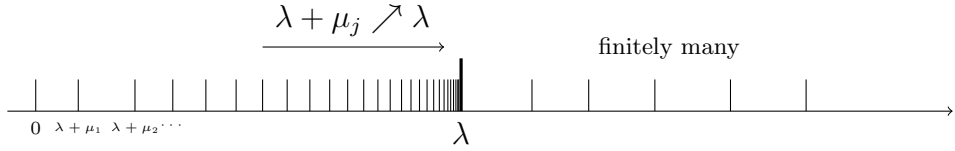

Consequently, for $\lambda\in J$, $\sigma_p(T_\lambda)$ is well-ordered.
Since $J$ is finite, $\Gamma_0:=
\bigcup_{\lambda\in J}\sigma_p(T_\lambda)$
is well-ordered. Hence $\sigma_p(T)=\sigma_{p}(T_0)\cup\Gamma_0$ is the countable well-ordered subset of $[0,\|T\|]$.

Since each $T_\lambda$ is diagonalizable and $H=\bigoplus_{\lambda\in\Lambda}H_\lambda$,
therefore the union of orthonormal eigenbases of the operators
$T_\lambda$ is an orthonormal basis for $H$ consisting of eigenvectors of $T$. Thus $T$ is diagonalizable. Since $T$ is positive and $\sigma_p(T)$ is a countable well-ordered subset of $[0,\|T\|]$, by Theorem \ref{+M_r class equiv}, it follows that $T\in\ma_r(H)_+$. 
\end{proof}

\begin{corollary}\label{Finite direct sum of AM Op}
    Let $A_j\in \am(H_j)$ for $j=1,\cdots,n$. Then $\bigoplus_{j=1}^n A_j\in \ma_r\left(\bigoplus_{j=1}^n H_j\right)$.
\end{corollary}

\begin{proof}
    Since $A_j\in \am(H_j)$, by \cite[Theorem 5.14]{AM Cha}, $A^*_jA_j\in \am(H_j)_+$ for $j=1,\cdots,n$. Thus by \cite[Theorem 5.7]{AM Cha}, $A^*_jA_j=\delta_j I_{H_j}-K_j+F_j$, where $\delta_j\geq 0$, $F_j\in \clf(H_j)_+$ and $K_j\in \clk(H_j)_+$ with $\|K_j\|\leq \delta_j$  for $j=1,\cdots,n$. Then
    \begin{align*}
        \bigoplus_{j=1}^n A_j^* A_j&=\bigoplus_{j=1}^n \delta_j I_{H_j} - \bigoplus_{j=1}^n K_j + \bigoplus_{j=1}^n F_j\\
        &=\sum_{j=1}^n \delta_j P_j - \widetilde{K} + \widetilde{F},
    \end{align*}
    where $P_j$ is the orthogonal projection of $H:=\bigoplus_{j=1}^n H_j$ onto $H_j$, $\widetilde{K}:=\bigoplus_{j=1}^n K_j\in\clk(H)_+$ and $\widetilde{F}:=\bigoplus_{j=1}^n F_j\in \clf(H)_+$. Also it is easy to observe that $P_j\widetilde{F}=\widetilde{F}P_j$, $P_j\widetilde{K}=\widetilde{K}P_j$ and $0\leq K_j\leq \delta_j I_{H_j}$ for $j=1,\cdots,n$. Therefore, by Theorem \ref{finite and compact perturbation suff} and Proposition \ref{equivalence}, the result follows.
\end{proof}

In general, the above result does not extend to infinite direct sums of
$\am$-operators.
\begin{remark}
 Let $H=\ell^2(\mathbb N)=\bigoplus_{j=1}^{\infty} H_j$, where $H_j=\bigvee\{e_j\}$.
For each $j\in\mathbb N$, define $T_j\in\clb(H_j)$ by
$$
T_j e_j=\left(1+\frac{1}{j}\right)e_j.
$$
Since $H_j$ is one-dimensional, each $T_j$ is $\am$. Consider the operator $T=\bigoplus_{j=1}^{\infty}T_j$. Then $\sigma_p(T)=\{1+\frac{1}{j} : j\in\N\}$, which does not have the least element. Thus $\sigma_p(T)$ is not well-ordered, and so $T\notin\ma_r\left(\bigoplus_{j=1}^\infty H_j\right)$
\end{remark}

\begin{lemma}\label{finite spectrum}
Every self-adjoint operator $T\in\clb(H)$ with finite spectrum belongs to $\ma_r(H)$.
\end{lemma}

\begin{proof}
 Let $\sigma(T)=\{\lambda_1,\ldots,\lambda_m\}$. Since $T$ is self-adjoint, by the
spectral theorem, we can write $T=\sum_{j=1}^m\lambda_jP_j$,
where $P_1,\ldots,P_m$ are mutually orthogonal projections with
$\sum_{j=1}^mP_j=I$. 

Let $M\neq\{0\}$ be a reducing subspace for $T$. Since each $P_j$ is a spectral projection of $T$, $M$ reduces $P_j$ for every $j$. Hence
$P_jM\subseteq M$. Since $M\neq\{0\}$, there exists $j$ such that
$P_jM\neq\{0\}$. Let $j_0$ be such that
$|\lambda_{j_0}|=\min\{|\lambda_j|:P_jM\neq\{0\}\}$. Choose a unit vector $x_0\in P_{j_0}M$. Then $\|Tx_0\|=|\lambda_{j_0}|$. Moreover, for every $x\in S_M$,
$$
\|Tx\|^2
=\sum_{j=1}^m|\lambda_j|^2\|P_jx\|^2
\geq |\lambda_{j_0}|^2\sum_{j:P_jM\neq\{0\}}\|P_jx\|^2
=|\lambda_{j_0}|^2.
$$
Thus $m(T_M)=|\lambda_{j_0}|=\|Tx_0\|$. Therefore, $T_M$ is minimum attaining. Since $M$ is arbitrary,
$T\in\ma_r(H)$.
\end{proof}

The following is a simple application of the spectral theorem. For the sake of completeness we include the proof here.
\begin{lemma}\label{self approx}
Let $T \in\clb(H)$ be a self-adjoint operator. Then $T$ can be approximated in operator norm by self-adjoint operators with finite spectrum.
\end{lemma}

\begin{proof}
Since $T$ is self-adjoint, its spectrum $\sigma(T)$ is a compact subset of $[-\|T\|, \|T\|]$. By the Spectral Theorem \cite[Theorem 20.2, Page 120]{KeheZhu}, there exists a unique spectral measure $E$ on the Borel $\sigma$-algebra of $\sigma(T)$ such that
$$
T = \int_{\sigma(T)} \lambda \, dE(\lambda).
$$
For a given $\varepsilon > 0$, cover the interval $[-\|T\|, \|T\|]$ with a finite collection of disjoint subintervals $I_1, I_2, \dots, I_n$ such that
$$
\text{diam}(I_k) < \frac{\varepsilon}{2} \quad \text{for all } k \in \{1, \dots, n\}.
$$
Let $A_k = I_k \cap \sigma(T)$ for $k = 1, \dots, n$, and define the index set of non-empty intersections
$$
J = \{ k \in \{1, \dots, n\} \mid A_k \neq \emptyset\}.
$$
Since $\{A_k\}_{k \in J}$ is a pairwise disjoint partition of $\sigma(T) = \bigcup_{k \in J} A_k$, it follows that $\{E(A_k)\}_{k \in J}$ is a family of nonzero, mutually orthogonal projections summing to the identity operator $I$. 

Now, for each $k \in J$, choose a point $\lambda_k \in A_k$, and define the operator 
$$
T_\varepsilon := \sum_{k \in J} \lambda_k E(A_k).
$$
Clearly, $T_\varepsilon$ is self-adjoint with finite spectrum $\sigma(T_\varepsilon) = \{\lambda_k \mid k \in J\} \subseteq \sigma(T)$. Therefore,
\begin{align*}
    \|T - T_\varepsilon\| &= \left\| \int_{\sigma(T)} \left(\lambda - \sum_{k \in J} \lambda_k \chi_{A_k}(\lambda)\right) dE(\lambda) \right\| \\
    &= \sup_{\lambda \in \sigma(T)} \left| \lambda - \sum_{k \in J} \lambda_k \chi_{A_k}(\lambda) \right|.
\end{align*}
Now for each $\lambda \in \sigma(T)$, there exists a unique $k \in J$ such that $\lambda \in A_k$. Consequently, we obtain
$$
\left| \lambda - \sum_{j \in J} \lambda_j \chi_{A_j}(\lambda) \right| = |\lambda - \lambda_k| \le \mathrm{diam}(A_k) < \frac{\varepsilon}{2}.
$$
Thus $\|T - T_\varepsilon\| \le \frac{\varepsilon}{2} < \varepsilon$. This completes the proof.
\end{proof}

We now prove the density theorem.
\begin{theorem}\label{densitythm}
The class $\ma_r(H)$ is norm dense in $\clb(H)$. 
\end{theorem}

\begin{proof}
    Let $T=V|T|$ be the polar decomposition of $T$. Then $V^*V$ is the orthogonal projection onto $N(T)^\perp$. Consider the restriction $|T|_{N(T)^\perp} : N(T)^\perp \longrightarrow N(T)^\perp$, which is a positive operator. Then by Lemma \ref{self approx}, for every $\varepsilon>0$, there exists a self-adjoint operator $A$ on $N(T)^\perp$ with finite spectrum such that $$\||T|_{N(T)^\perp}-A\|<\varepsilon.$$
    With respect to the decomposition $H=N(T)\oplus N(T)^\perp$, write $V^*V=\begin{bmatrix}
        0 & 0\\
        0 & I_{N(T)^\perp}
    \end{bmatrix}$, and consider the operator $B:=\begin{bmatrix}
        0 & 0\\
        0 & A
    \end{bmatrix}$. Set $S:=VB$. Then $S^*S=B^*V^*VB=\begin{bmatrix}
        0 & 0 \\
        0 & A^2
    \end{bmatrix}$, and hence $|S|=\begin{bmatrix}
        0 & 0 \\
        0 & |A|
    \end{bmatrix}$. Since $\sigma(|A|)=\{|\lambda|:\lambda\in\sigma(A)\}$, the operator $|A|$ has finite spectrum. Consequently, $|S|$ is a positive operator with finite spectrum. By Lemma \ref{finite spectrum}, it follows that $|S|\in\ma_r(H)_+$. Hence Proposition \ref{equivalence}(ii)
    gives $S\in\ma_r(H)$. 
    
    Therefore,
    $$\|T-S\|=\|V(|T|-B)\|\leq \|V\|\||T|-B\|=\||T|_{N(T)^\perp}-A\|<\varepsilon.$$
    Thus $\ma_r(H)$ is dense in $\clb(H)$.
 \end{proof}
Next, we strengthen the above density theorem. In fact, we prove that the operator which approximates the given bounded operator has an invariant half-space. Before proving this, we recall the following definition of half-space.
\begin{definition}\label{halfspace}
    A closed subspace $M$ of a Banach space $X$ is said to be a half-space if both $M$ and the quotient space $X/M$ are infinite dimensional. 
\end{definition}

We shall use the following result on finite rank perturbations of operators
with finite spectrum.

\begin{proposition}\cite[Proposition 2.11]{Beta Jot}\label{prop:finite-rank-finite-spectrum}
Let $S \in \mathcal{B}(H)_+$ have finite spectrum, and let
$C \in \mathcal{F}(H)$ be self-adjoint. Then $\sigma(S+C)$ is a finite set.
\end{proposition}

\begin{remark}
In \cite[Proposition 2.11]{Beta Jot}, the spectrum of $S$ is written as
$\{\alpha_i\}_{i=1}^{n}$ with $\alpha_1 \geq \alpha_2 \geq \cdots \geq \alpha_n \geq 0$.
This ordering is only a labeling of the finitely many points of $\sigma(S)$ and
plays no role in the statement. Hence the result applies to every positive
operator with finite spectrum.
\end{remark}

\begin{corollary}\label{density-halfspace}
The set of operators in $\ma_r(H)$ having a nontrivial invariant
half-space is norm dense in $\mathcal B(H)$.
\end{corollary}

\begin{proof}
Let $T\in\mathcal B(H)$, and choose $\varepsilon>0$. By
Theorem \ref{densitythm}, there exists $S\in\ma_r(H)$ with $\sigma(|S|)$ finite such that
$$
\|T-S\|<\frac{\varepsilon}{2}.
$$
By \cite[Theorem 4.3]{Tcaciuc}, there exists a finite rank operator
$F$ with $\|F\|<\frac{\varepsilon}{2}$
such that $S+F$ has a nontrivial invariant half-space.

We claim that $S+F\in\ma_r(H)$. Indeed,
$$
(S+F)^*(S+F)=|S|^2+ \widetilde{F},
$$
where $\widetilde{F}:=S^*F+F^*S+F^*F\in\mathcal F(H)_{\mathrm{sa}}$.

Since $|S|$ has finite spectrum, so does $|S|^2$, by the spectral mapping theorem. Hence by Proposition \ref{prop:finite-rank-finite-spectrum}, $(S+F)^*(S+F)$ has finite spectrum. Therefore, by Lemma \ref{finite spectrum}, it follows that $(S+F)^*(S+F)\in\ma_r(H)_+$.
Proposition \ref{equivalence}(ii) now yields
$S+F\in\ma_r(H)$.

Finally, we obtain
$$
\|T-(S+F)\|
\leq \|T-S\|+\|F\|
<\frac{\varepsilon}{2}+\frac{\varepsilon}{2}
=\varepsilon.$$
Thus the operators in $\ma_r(H)$ having an  invariant
half-space are norm dense in $\mathcal B(H)$.
\end{proof}

\section{Normal Operators in $\ma_r(H)$}

An operator $T\in\clb(H)$ is said to be quasinormal if $T(T^*T)=(T^*T)T$, equivalently $T|T|=|T|T$. The class of quasinormal operators includes normal operators. We refer to \cite{Furuta} for further details. In this section, we show that if $T\in\ma_r(H)$ is quasinormal, then $|T|\in\ma_r(H)_+$. For normal operators, this leads to a
structural characterization. Also we show that a normal operator in $\ma_r(H)$ has closed range. Then we provide the
characterization for multiplication operators on $L^2(\T)$ in terms of its symbol function. Finally, we give a characterization for which an operator $T\in\clb(H)$ satisfies the relation $\red_T=\red_{|T|}$.

\begin{proposition}\label{quasinormal mod T}
Let $T\in\ma_r(H)$ be quasinormal. Then $|T|\in\ma_r(H)_+$.
\end{proposition}

\begin{proof}
Since $T\in\ma_r(H)$ and $H\in\red_T$, by Lemma \ref{ma set}, there exists a nonzero $x\in H$ such that $|T|x=m(T)x$. Choosing $x$ as unit vector, in the similar way of Theorem \ref{diagonalizable}, it can be shown that $|T|$ has a maximal orthonormal set of eigenvectors. Let $\mathcal{E}$ be a maximal orthonormal set consisting of eigenvectors
of $|T|$, and assume that $\widetilde{\cle}=\bigvee \cle$.

We claim that $\widetilde{\cle}=H$, that is, $|T|$ is diagonalizable. To arrive at a contradiction, assume that $\widetilde{\cle}^\perp\neq \{0\}$.

Since $T$ is quasinormal, we have $T|T|=|T|T$. Let $x\in\cle$ and suppose that $\lambda$ is the associated eigenvalue, that is,  $|T|x=\lambda x$. Then
$$
|T|(Tx)=T|T|x=\lambda Tx.
$$
Thus $Tx\in N(|T|-\lambda I)$. Similarly, using $T^*|T|=|T|T^*$, we get
$$
T^*x\in N(|T|-\lambda I).
$$
Therefore, the closed span $\widetilde{\cle}$ reduces $T$. Consequently, $\widetilde{\cle}^\perp\in\red_T$. Consider the restriction $S:=T|_{\widetilde{\cle}^\perp}$. Then $S\in\ma(\widetilde{\cle}^\perp)$. Hence there exists a unit vector $y\in\ \widetilde{\cle}^\perp$ such that $|S|y=m(S)y$.  Since $|S|=|T_{\widetilde{\cle}^\perp}|=|T|_{\widetilde{\cle}^\perp}$ by \cite[Lemma 3.4]{betaclass}, it follows that $|T|y=m(S)y$.  Thus $y$ is an eigenvector of $|T|$ lying in $\widetilde{\cle}^\perp$, which contradicts the maximality of $\cle$. Therefore, $\widetilde{\cle}^\perp=\{0\}$, and hence $\widetilde{\cle}=H$. Consequently, $|T|$ is diagonalizable.

\smallskip
Let $\Lambda:=\sigma_p(|T|)$. Since $|T|$ is positive, eigenspaces corresponding to distinct
eigenvalues are mutually orthogonal. Thus we can write $H=\bigoplus_{\lambda\in\Lambda}
N(|T|-\lambda I)$.

Now we show that $\Lambda$ is well-ordered. If not, then there exists a strictly
decreasing sequence
$$
\lambda_1>\lambda_2>\cdots,
\qquad \lambda_n\in\Lambda.
$$
Let $M:=\bigoplus_{n=1}^{\infty}
N(|T|-\lambda_n I)$. Since each eigenspace $N(|T|-\lambda_nI)$ reduces $T$, it follows that $M\in\red_T$. For every $x\in M$ with $\|x\|=1$, write $x=\sum_{n=1}^{\infty}x_n$, where $x_n\in N(|T|-\lambda_nI)$. Then we have
$$
\|Tx\|^2
=\||T|x\|^2
=\sum_{n=1}^{\infty}\lambda_n^2\|x_n\|^2 > (\inf_{n\geq 1} \lambda_n)^2.
$$
On the other hand, by Lemma \ref{m(T) for direct sum operators}, $m(T_M)=m(|T|_M)=\inf_{n\geq 1} \lambda_n$. Thus $\|Tx\|>m(T_M)$ for every unit vector $x\in M$. It follows that $T_M\notin \ma(M)$, which is a contradiction. 

Therefore, $\Lambda$ is a well-ordered subset of $[0,\|T\|]$.
Since $H$ is separable, $\Lambda$ is countable. Hence by Theorem \ref{+M_r class equiv}, it follows that $|T|\in\ma_r(H)_+$.
\end{proof}

\begin{remark}
    Let $T\in\clb(H)$ be a quasinormal operator. Then $T\in\ma_r(H)$ if and only if $|T|\in\ma_r(H)_+$.
\end{remark}

\begin{corollary}\label{normal mod T}
    Let $T\in\clb(H)$ be a normal operator. Then $T\in\ma_r(H)$ if and only if $|T|\in\ma_r(H)_+$.
\end{corollary}

\begin{proposition}\label{Normal T and T*}
Let $T$ be normal. Then $T\in\ma_r(H)$ if and only if $T^*\in\ma_r(H)$.
\end{proposition}

\begin{proof}
    Since $T$ is normal, $N(T)=N(T^*)$. The result follows by Proposition \ref{WEP T and T* eqv}.
\end{proof}

\begin{proposition}\label{normal closed range}
Let $T\in\ma_r(H)$ be normal. Then $T$ has closed range.
\end{proposition}

\begin{proof}
For normal $T$, $N(T)=N(T^*)$, so $N(T)^\perp$ reduces $T$. The restriction
$T_0=T|_{N(T)^\perp}$ is injective and minimum attaining. Hence its minimum
modulus is nonzero, otherwise attainment of $m(T_0)=0$ would contradict
injectivity. Thus $T_0$ is bounded below, and $R(T)=R(T_0)$ is closed.
\end{proof}

\begin{proposition}
Let $T\in\ma_r(H)$ be normal, and let $S$ be an isometry. Then $ST,\, ST^*,\, TS^*\in\ma_r(H)$.
\end{proposition}

\begin{proof}
Since $T\in\ma_r(H)$ is normal, Proposition \ref{Normal T and T*} gives
$T^*\in\ma_r(H)$. By Corollary \ref{normal mod T}, we get $|T|,\ |T^*|\in\ma_r(H)_+$. Hence by Proposition \ref{equivalence}, it follows that $T^*T,\ TT^*\in\ma_r(H)_+$.

Let $X=ST$. Since $S$ is an isometry, $S^*S=I$, and hence
$$
X^*X=T^*S^*ST=T^*T\in\ma_r(H)_+.
$$
Therefore, by Proposition \ref{equivalence}, we obtain $X=ST\in\ma_r(H)$.

\smallskip
Next, assume that $Y=ST^*$. Then
$$
Y^*Y=TS^*ST^*=TT^*\in\ma_r(H)_+.
$$
Again, by Proposition \ref{equivalence}, we obtain $Y=ST^*\in\ma_r(H)$.

\smallskip
It remains to prove that $TS^*\in\ma_r(H)$. Let $Z=TS^*$. Since $T^*T\in\ma_r(H)_+$, by Theorem \ref{+M_r class equiv},
there exist a countable well-ordered set
$\Lambda\subseteq[0,\|T^*T\|]$ and mutually orthogonal projections
$(P_\lambda)_{\lambda\in\Lambda}$ with $\sum_{\lambda\in\Lambda}P_\lambda=I$ such that $T^*T=\sum_{\lambda\in\Lambda}\lambda P_\lambda$ in the strong operator topology.

Now
$$
Z^*Z=ST^*TS^*=\sum_{\lambda\in\Lambda}\lambda Q_\lambda, \text{ where }
Q_\lambda:=SP_\lambda S^*.
$$
Since $S^*S=I$, each $Q_\lambda$ is an orthogonal projection, and
$$
Q_\lambda Q_\mu
=SP_\lambda S^*SP_\mu S^*
=SP_\lambda P_\mu S^*
=0 \text{ for } \lambda\neq\mu.
$$
Moreover, observe that
$$
\sum_{\lambda\in\Lambda}Q_\lambda
=S\left(\sum_{\lambda\in\Lambda}P_\lambda\right)S^*
=SS^*.
$$

Now we distinguish two cases.

\smallskip
\noindent\textbf{Case I: $0\notin\Lambda$.}
Set $Q_0:=I-SS^*$. Then $Q_0$ is an orthogonal projection and $Q_0Q_\lambda=(I-SS^*)SP_\lambda S^*=0$ for every $\lambda\in\Lambda$. Hence
$(Q_0,Q_\lambda)_{\lambda\in\Lambda}$ is a mutually orthogonal family of
projections and
$$
Q_0+\sum_{\lambda\in\Lambda}Q_\lambda
=I-SS^*+SS^*
=I.
$$
Thus we can write
$$
Z^*Z
=0.Q_0+ \sum_{\lambda\in\Lambda}\lambda Q_\lambda.
$$
Since $\{0\}\cup\Lambda$ is countable and well-ordered, Theorem \ref{+M_r class equiv} yields $Z^*Z\in\ma_r(H)_+$.

\smallskip
\noindent\textbf{Case II: $0\in\Lambda$.}
Define $\widetilde Q_0:=SP_0S^*+I-SS^*$ and $\widetilde Q_\lambda:=SP_\lambda S^*$ for $\lambda\in\Lambda\setminus\{0\}$.
Since $(I-SS^*)SP_0S^*=0$, we get
$\widetilde Q_0$ is an orthogonal projection. Furthermore,
the family $(\widetilde Q_\lambda)_{\lambda\in\Lambda}$ is mutually
orthogonal and
$$
\sum_{\lambda\in\Lambda}\widetilde Q_\lambda
=I-SS^*+\sum_{\lambda\in\Lambda}SP_\lambda S^*
=I.
$$
Since the coefficient corresponding to $\widetilde Q_0$ is zero,
we can write $Z^*Z
=\sum_{\lambda\in\Lambda}\lambda\widetilde Q_\lambda$.
As $\Lambda$ is countable and well-ordered, by Theorem \ref{+M_r class equiv}, we obtain $Z^*Z\in\ma_r(H)_+$.

In either case, Proposition \ref{equivalence} implies $Z=TS^*\in\ma_r(H)$.
\end{proof}

Next, we give a structural characterization of normal operators in
$\mathcal{M}_r(H)$ in terms of orthogonal direct sums of scalar
multiples of unitary operators.

\begin{theorem}\label{normalrepn}
    Let $T\in\clb(H)$ be a normal operator. Then the following are equivalent:
    \begin{enumerate}
        \item $T\in\ma_r(H)$;

        \item There exist a countable family of pairwise orthogonal reducing subspaces $(H_a)_{a\in\Lambda}$ for $T$ and unitary operators $U_a\in\clb(H_a)$ such that $$H=\bigoplus_{a\in\Lambda} H_a \text{ and }\; T=\bigoplus_{a\in\Lambda}a\,U_a,$$ where $\Lambda:=\sigma_p(|T|)\subseteq[0,\|T\|]$ is a countable well-ordered set.
    \end{enumerate}
\end{theorem}

\begin{proof}
   (i)$\implies$(ii): Since $T$ is normal and $T\in\ma_r(H)$, Corollary \ref{normal mod T} gives $|T|\in\ma_r(H)_+$. Therefore, by Theorem \ref{+M_r class equiv}, there exists a countable well-ordered set
$\Lambda\subseteq[0,\infty)$ and mutually orthogonal projections
$(P_a)_{a\in\Lambda}$ with $\sum_{a\in\Lambda}P_a=I$ such that
$|T|=\sum_{a\in\Lambda}aP_a$ in the strong operator topology. It is easy to observe that $\Lambda=\sigma_p(|T|)$. Let $H_a:=N(|T|-aI)$ for $a\in\Lambda$. Then it follows that $H=\bigoplus_{a\in\Lambda} H_a$.

Normality of $T$ gives us $T|T|=|T|T$ and $T^*|T|=|T|T^*$. Consequently, it follows that for each $a\in\Lambda$, $H_a$ reduces $T$.

\smallskip
Let $T=V|T|$ be the polar decomposition of $T$. 

\smallskip
\noindent\textbf{Case 1: $a=0$.} In this case, $H_0=N(|T|)=N(T)$, and hence $T|_{H_0}=0$. In this case, we can define $U_0=I|_{H_0}$.

\smallskip
\noindent\textbf{Case 2: $a>0$.} Then for every $x\in H_a$, we have $Tx=V|T|x=aVx$. Moreover, $H_a\subseteq N(T)^\perp$, so $V|_{H_a}$ is an isometry on $H_a$. Since $T|_{H_a}$ is normal, we obtain $V|_{H_a}$ is a normal isometry. Consequently, $V|_{H_a}$ is unitary. Taking $U_a:=V|_{H_a}$, we can write $$T|_{H_a}=aU_a,\quad a>0.$$
Therefore, combining above two cases we obtain
$$
T=\bigoplus_{a\in\Lambda}a\,U_a.
$$

(ii)$\implies$(i): Assume that (ii) holds. Then by Theorem \ref{Red min sufficient}, it follows that $|T|\in\ma_r(H)_+$. By Corollary \ref{normal mod T}, we obtain $T\in\ma_r(H)$.
\end{proof}

Now we present an equivalent characterization of normal operators
in $\mathcal{M}_r(H)$ in terms of mutually orthogonal projections
and a family of partial isometries.
\begin{theorem}\label{general case}
Let $T\in\clb(H)$ be normal. Then the following are equivalent:
\begin{enumerate}
    \item $T\in\ma_r(H)$;

    \item There exist a countable well-ordered set
    $\Lambda\subseteq[0,\|T\|]$, a family of mutually orthogonal
    projections $(P_\lambda)_{\lambda\in\Lambda}$, and a family of
    partial isometries $(V_\lambda)_{\lambda\in{\Lambda\setminus \{0\}}}$ with $\sum_{\lambda\in\Lambda}P_\lambda=I$ and $V_\mu^*V_\lambda
    =\delta_{\mu\lambda}P_\lambda$ for $\lambda,\mu\in\Lambda\setminus\{0\}$ such that
    \begin{equation}
    T=\sum_{\lambda\in\Lambda\setminus\{0\}}\lambda V_\lambda
    \end{equation}
    in the strong operator topology.
\end{enumerate}
\end{theorem}

\begin{proof}
(i)$\implies$(ii): If $T\in\ma_r(H)$, by Corollary \ref{normal mod T}, we have $|T|\in\ma_r(H)_+$. Hence by Theorem \ref{+M_r class equiv}, there exist a bounded countable well-ordered set
$\Lambda\subseteq[0,\|T\|]$ and mutually orthogonal projections
$(P_\lambda)_{\lambda\in\Lambda}$ with
$\sum_{\lambda\in\Lambda}P_\lambda=I$ such that $|T|=\sum_{\lambda\in\Lambda}\lambda P_\lambda$ in the strong operator topology. 

Let $T=W|T|$ be the polar decomposition of $T$. For $\lambda>0$, observe that
$R(P_\lambda)\subseteq N(T)^\perp$. Define $V_\lambda:=WP_\lambda$ for $\lambda\in\Lambda\setminus\{0\}$. Then for each $\lambda>0$, $V_\lambda$ is a partial isometry on $H$ with initial space $R(P_\lambda)$. 

\smallskip
Since $W^*W$ is an orthogonal projection onto $N(T)^\perp$, it follows that 
\begin{equation*}
V_\mu^*V_\lambda=P_\mu W^*WP_\lambda=P_\mu P_\lambda=\delta_{\mu\lambda}P_\lambda\, \text{ for }\, \lambda,\mu\in\Lambda\setminus\{0\}.
\end{equation*}
Therefore, we obtain $T=W|T|=\sum_{\lambda\in\Lambda\setminus\{0\}}\lambda V_\lambda$.
\medskip

(ii)$\implies$(i): Assume that (ii) holds. Then 
\begin{align*}
T^*T&=\left(\sum_{\mu\in\Lambda\setminus\{0\}}\mu V_\mu\right)^*
\left(\sum_{\lambda\in\Lambda\setminus\{0\}}\lambda V_\lambda\right)\\
&=\sum_{\mu,\lambda\in\Lambda\setminus\{0\}}
\mu\lambda V_\mu^*V_\lambda\\
&=\sum_{\lambda\in\Lambda\setminus\{0\}}\lambda^2P_\lambda.
\end{align*}
By Theorem \ref{+M_r class equiv} and Proposition \ref{equivalence}, it follows that $T\in\ma_r(H)$.
\end{proof}

Next, we provide a necessary and sufficient condition for the multiplication operator $M_\vp$ on $L^2(\T)$ induced by $\vp\in L^\infty(\T)$ to be in $\ma_r(L^2(\T))$.

\begin{theorem}\label{thm:multiplication_op}
Let $\varphi \in L^\infty(\mathbb{T})$. Then the following statements are equivalent:
\begin{enumerate}
    \item $M_\varphi \in \mathcal{M}_r(L^2(\mathbb{T}))$;
    \item There exist a countable well-ordered set $\Lambda \subseteq [0,\|\vp\|_{\infty}]$, a family $\{E_\lambda\}_{\lambda \in \Lambda}$ of pairwise essentially disjoint measurable subsets of $\,\mathbb{T}$ satisfying $\mu(E_\lambda) > 0$ for every $\lambda \in \Lambda$ with $\mu\left(\mathbb{T} \setminus \bigcup_{\lambda \in \Lambda} E_\lambda\right) = 0$, and a function $\omega \in L^\infty(\mathbb{T})$ with $|\omega| = 1$ a.e. on $\mathbb{T}$ such that
    $$
    \varphi = \omega \sum_{\lambda \in \Lambda} \lambda \chi_{E_\lambda} \quad \text{a.e. on } \mathbb{T}.
    $$
\end{enumerate}
\end{theorem}

\begin{proof}
(i)$\implies$(ii): Let $M_\varphi \in \mathcal{M}_r(L^2(\mathbb{T}))$. Since $M_\varphi$ is a normal operator, by Corollary \ref{normal mod T}, $M_{|\varphi|} \in \mathcal{M}_r(L^2(\mathbb{T}))_+$. Then by Theorem \ref{Red Min Necessary}, we have
$$
M_{|\varphi|} = \sum_{\lambda \in \Lambda} \lambda P_\lambda,
$$
where $\Lambda := \sigma_p(M_{|\varphi|}) \subseteq [0, \|\vp\|_{\infty}]$ is a countable well-ordered set, and $(P_\lambda)_{\lambda \in \Lambda}$ is a family of mutually orthogonal projections satisfying $\sum_{\lambda \in \Lambda} P_\lambda = I$ in the strong operator topology.

\smallskip
Clearly, for each $\lambda \in \Lambda$, $P_\lambda$ is the orthogonal projection onto the eigenspace $H_\lambda:= \ker(M_{|\varphi|} - \lambda I)$. 
Define the measurable set $E_\lambda = \{x \in \mathbb{T} : |\varphi(x)| = \lambda\}$. For any $f \in L^2(\mathbb{T})$,
\begin{align*}
f \in H_\lambda &\iff (|\varphi| - \lambda)f = 0 \quad \text{a.e. on } \mathbb{T} \\
&\iff f = 0 \quad \text{a.e. on } \mathbb{T} \setminus E_\lambda \\
&\iff f \in L^2(E_\lambda).
\end{align*}
Thus $R(P_\lambda) = H_\lambda = L^2(E_\lambda) = R(M_{\chi_{E_\lambda}})$, which implies $P_\lambda = M_{\chi_{E_\lambda}}$ for each $\lambda \in \Lambda$.

Since $P_\lambda \neq 0$, we have $\mu(E_\lambda) > 0$. Further, $P_\lambda P_\nu = 0$ for $\lambda \neq \nu$ implies that $\mu(E_\lambda \cap E_\nu) = 0$, that is, $E_\lambda\cap E_\nu=\emptyset$ a.e. Consider $A := \mathbb{T} \setminus \bigcup_{\lambda \in \Lambda} E_\lambda$ and $f := \chi_A \in L^2(\mathbb{T})$. Since $\sum_{\lambda \in \Lambda} P_\lambda = I$, we have
$$
\chi_A = \sum_{\lambda \in \Lambda} M_{\chi_{E_\lambda}} \chi_A = \sum_{\lambda \in \Lambda} \chi_{E_\lambda \cap A} = 0 \quad \text{a.e. on } \mathbb{T},
$$
and hence $\mu(A) = 0$. Now observe that
$$
M_{|\varphi|} = \sum_{\lambda \in \Lambda} \lambda M_{\chi_{E_\lambda}} = M_{\sum_{\lambda \in \Lambda} \lambda \chi_{E_\lambda}}.
$$
Therefore, $|\varphi| = \sum_{\lambda \in \Lambda} \lambda \chi_{E_\lambda}$ a.e. on $\mathbb{T}$. By the polar decomposition of $\varphi$, there exists $\omega \in L^\infty(\mathbb{T})$ satisfying $|\omega| = 1$ a.e. on $\mathbb{T}$ such that
$$
\varphi = \omega |\varphi| = \omega \sum_{\lambda \in \Lambda} \lambda \chi_{E_\lambda} \quad \text{a.e. on } \mathbb{T}.
$$

(ii)$\implies$(i): Assume that (ii) holds. Then $|\varphi| = \sum_{\lambda \in \Lambda} \lambda \chi_{E_\lambda}$, and hence $M_{|\varphi|} = \sum_{\lambda \in \Lambda} \lambda M_{\chi_{E_\lambda}}$. Since $\{E_\lambda\}_{\lambda \in \Lambda}$ are pairwise essentially disjoint measurable subsets of $\mathbb{T}$ with positive measure, $M_{\chi_{E_\lambda}}$ are nonzero mutually orthogonal projections. Also, $\mu\left(\mathbb{T} \setminus \bigcup_{\lambda \in \Lambda} E_\lambda\right) = 0$ implies $\sum_{\lambda \in \Lambda} \chi_{E_\lambda} = 1$ a.e. on $\mathbb{T}$. Consequently, we obtain $\sum_{\lambda \in \Lambda} M_{\chi_{E_\lambda}} = I$ in the strong operator topology. By Theorem \ref{Red min sufficient}, it follows that $M_{|\varphi|} \in \mathcal{M}_r(L^2(\mathbb{T}))_+$. Hence by Proposition \ref{equivalence}(ii), we conclude that $M_\varphi \in \mathcal{M}_r(L^2(\mathbb{T}))$.
\end{proof}

\medskip
In general, an operator $T \in \clb(H)$ and its modulus $|T|$ need not have the same collection of reducing subspaces. We next characterize the condition $\red_T = \red_{|T|}$ in terms of the von Neumann algebra generated by $|T|$.

We recall some basic notions concerning von Neumann algebras that will be used
in the sequel. For a subset $\mathcal{S}$ of $\mathcal{B}(H)$, the
\emph{commutant} of $\mathcal{S}$ is defined by
$$
\mathcal{S}'
=
\{A\in\mathcal{B}(H): AS=SA\ \text{for every } S\in\mathcal{S}\}.
$$
The \emph{bicommutant} of $\mathcal{S}$ is given by $\mathcal{S}''=(\mathcal{S}')'$.
In particular, for $T\in \clb(H)$, we write
$$
\{T,T^*\}'
=\{A\in \clb(H): AT=TA
\text{ and }AT^*=T^*A\},
$$
and $\{T,T^*\}''=\left(\{T,T^*\}'\right)'$.

A unital $*$-subalgebra $\mathcal{A}$ of $\mathcal{B}(H)$ is called a
\emph{von Neumann algebra} if it is closed in the weak operator topology.
By the von Neumann bicommutant theorem, a unital $*$-subalgebra
$\mathcal{A}$ of $\mathcal{B}(H)$ is a von Neumann algebra if and only if $\mathcal{A}=\mathcal{A}''$.
The von Neumann algebra generated by $T$ is denoted by
$W^*(T)$ and is given by
$W^*(T):=\{T,T^*\}''$.
If $T$ is normal, then Fuglede's theorem gives $\{T,T^*\}'=\{T\}'$ and consequently $W^*(T)=\{T\}''$. For further details, see \cite{Murphy,KeheZhu}.

We shall use these notions to characterize the condition
$\mathcal{R}_T=\mathcal{R}_{|T|}$ in terms of the partial isometry
appearing in the polar decomposition of $T$.

\begin{theorem}\label{T-modT=red}
    Let $T \in \clb(H)$ with polar decomposition $T = V|T|$. Then the following are equivalent:
\begin{enumerate}
    \item $\mathcal R_T=\mathcal R_{|T|}$;
    \item $V\in W^*(|T|)=\{|T|\}''$;
    \item $W^*(T)=W^*(|T|)$.
\end{enumerate}
\end{theorem}

\begin{proof}
    (i)$\implies$(ii): Assume that $\red_T=\red_{|T|}$. If $P\in \{|T|\}^{'}$ is any orthogonal projection on $H$, then $R(P)$ reduces $|T|$ as well as $T$. Then we have
    \begin{align*}
        PT&=TP, \text{ that is, } PV|T|=V|T|P \text{ or, } PV|T|=VP|T|,
    \end{align*}
    and hence $(PV-VP)|T|=0$. Therefore, $PV=VP$ on $R(|T|)$, and it can be extended to $\overline{R(|T|)}=N(|T|)^\perp$. 

    For $x\in N(|T|)$, it follows that $|T|Px=P|T|x=0$. Hence $Px\in N(|T|)=N(V)$. Consequently, $VPx=0$ and $PVx=0$. Thus we obtain $PV=VP$ on $H$.

    We have shown that $V$ commutes with every orthogonal projection $P\in\{|T|\}'$.
    Since $\{|T|\}^{'}$ is a von Neumann algebra, by \cite[Theorem 20.3, Page 122]{KeheZhu}, every element of $\{|T|\}'$ is a strong operator limit of finite linear combinations of orthogonal projections from $\{|T|\}'$. Hence $V$ commutes with every $S\in\{|T|\}'$. Thus $V\in \{|T|\}^{''}=W^{*}(|T|)$.
\medskip

    (ii)$\implies$(i): Let $V\in W^*(|T|)=\{|T|\}^{''}$, the von Neumann algebra generated by $|T|$. Then $VS=SV$ for all $S\in \{|T|\}^{'}$. Let $M\in\red_{|T|}$ and $P_M$ be the orthogonal projection onto $M$. Then $P_M |T|=|T|P_M$, and further we get $P_M\in \{|T|\}^{'}$. Therefore, $VP_M=P_M V$ and $V^*P_M=P_M V^*$. Consequently, $M$ reduces $V$. From the polar decomposition of $T$, it follows that $M$ reduces $T$. Thus we get $\red_{|T|}\subseteq \red_T$. Also we already have $\red_T\subseteq \red_{|T|}$. Combining, we obtain $\red_T=\red_{|T|}$.

    \medskip
    (ii)$\implies$(iii): Let  $V\in W^*(|T|)$. Since
$T=V|T|$, we have $T\in W^*(|T|)$. It follows that $W^*(T)\subseteq W^*(|T|)$.
On the other hand, $|T|=(T^*T)^{1/2}\in W^*(T)$,
and hence $W^*(|T|)\subseteq W^*(T)$.
Therefore, we obtain $W^*(T)=W^*(|T|)$.

\medskip
(iii)$\implies$(i): Taking commutants in $W^*(T)=W^*(|T|)$ gives $W^*(T)'=W^*(|T|)'$.
The orthogonal projections belonging to these commutants correspond precisely to the reducing subspaces of $T$ and $|T|$, respectively. Hence $\mathcal R_T=\mathcal R_{|T|}$.
\end{proof}

\begin{corollary}\label{normality from red T=|T|}
Let $T\in \clb(H)$ satisfy $\mathcal R_T=\mathcal R_{|T|}$. Then $T$ is normal.
\end{corollary}

\begin{proof}
By Theorem \ref{T-modT=red}, we have $V\in W^*(|T|)=\{|T|\}''$. Since $|T|$ is self-adjoint, the von Neumann algebra
$W^*(|T|)$ is abelian. Consequently, $V$ is normal and
commutes with $|T|$. 
Therefore,
$$
T^*T=|T|V^*V|T|=|T|^2V^*V=|T|^2VV^*=V|T|^2V^*=TT^*.
$$
Hence $T$ is normal.
\end{proof}

\begin{remark}
The converse of Corollary \ref{normality from red T=|T|} does not hold in general.
Let $U$ be the bilateral shift on $\ell^2(\mathbb Z)$, defined by $Ue_n=e_{n+1}$ for $n\in\mathbb{Z}$. Then $U$
is unitary, hence normal, and $|U|=(U^*U)^{1/2}=I$. Every closed subspace
reduces $I$, so $\mathcal R_{|U|}$ is the set of all closed subspaces of $\ell^2(\mathbb{Z})$. However,
$M=\overline{\operatorname{span}}\{e_0\}$ is not invariant under $U$, because
$Ue_0=e_1\notin M$. Thus $M\notin\mathcal R_U$ and $\mathcal R_U\neq\mathcal R_{|U|}$.
Equivalently, in the polar decomposition $U=V|U|$ we have $V=U$, while
$W^*(|U|)=\{I\}''=\mathbb C I$, and $U\notin\mathbb C I$. So condition (ii) of
Theorem \ref{T-modT=red} fails.
\end{remark}

\section*{Declarations}

\subsection*{Conflict of Interest}
The authors declare that there is no conflict of interest relevant to this article.

\subsection*{Funding}
Not Applicable.

\subsection*{Author Contributions} 
Both authors contributed equally to the conceptualization and preparation of the manuscript.

\subsection*{Data Availability}
Not Applicable.

\section*{Acknowledgements}
The authors used ChatGPT (OpenAI) to assist with improving the language and polishing the grammar. All mathematical ideas, results, proofs, and conclusions are the authors' own, and the authors take full responsibility for the contents of the manuscript.

\end{document}